\documentclass[reqno]{amsart}[12pt]

\usepackage{hyperref}
\hypersetup{colorlinks,linkcolor={blue},citecolor={blue},urlcolor={red}}  
\usepackage[paper=letterpaper,margin=.9in]{geometry} 
\usepackage[utf8]{inputenc}
\usepackage{lmodern}
\usepackage{array}
\usepackage{adjustbox}
\usepackage{graphicx}
\usepackage{amssymb, amsfonts, amsthm, amsmath}
\usepackage[english]{babel}
\usepackage{float}
\usepackage{comment}
\usepackage{tikz}
\usepackage{bbm}
\usepackage{enumerate}
\usetikzlibrary{decorations.markings}
\usetikzlibrary{calc,angles,positioning}
\usetikzlibrary{decorations.pathreplacing}
\usepackage{chngcntr}
\usepackage{apptools}
\theoremstyle{plain}
\newtheorem{theorem}{Theorem}[section]

\newtheorem{defin}[theorem]{Definition}
\newtheorem{prop}[theorem]{Proposition}
\newtheorem{cor}[theorem]{Corollary}
\newtheorem{lemma}[theorem]{Lemma}

\newtheorem{remark}[theorem]{Remark}

\newcommand\numberthis{\addtocounter{equation}{1}\tag{\theequation}}
\newcommand{\HH}{\mathsf{H}}
\newcommand{\vt}{v}

\newcommand{\FF}{\mathcal{F}}

\newcommand{\p}{\overset{p}{\to}}
\renewcommand{\b}{\mathsf{b}}
\begin{document}
\numberwithin{equation}{section}
\title{The stochastic telegraph equation limit of the stochastic higher spin six vertex model}
\author{Yier Lin}
\address{Y.\ Lin,
	Department of Mathematics, Columbia University,
	2990 Broadway, New York, NY 10027}
\email{yl3609@columbia.edu}
\begin{abstract}
In this paper, we
prove that the stochastic telegraph equation arises as a scaling limit of the stochastic higher spin six vertex (SHS6V)
model with general spin
$I/2, J/2$. This extends results of Borodin and Gorin which focused on the $I=J=1$ six vertex case  and demonstrates the universality of the stochastic telegraph equation in this context. We also provide a functional extension of the central limit theorem obtained
in \cite[Theorem 6.1]{BG19}. The main idea is to generalize the four point relation established in \cite[Theorem 3.1]{BG19}, using fusion.
\end{abstract}
\maketitle
\section{Introduction}

\subsection{Telegraph equation and stochastic telegraph equation}
The \emph{telegraph equation} is a hyperbolic PDE given by
\begin{equation}\label{eq:telegraph}
\begin{cases}
u_{XY} (X, Y)  + \beta_1 u_Y (X, Y) + \beta_2 u_X (X, Y) = f(X, Y),\\
u(X, 0) = \chi(X), \quad u(0, Y) = \psi(Y),
\end{cases}
\end{equation}
where the functions $\chi, \psi \in C^1$ satisfy $\chi(0) = \psi(0)$. When $f$ is a deterministic function, the equation \eqref{eq:telegraph} is a classical object, see \cite[Chapter V]{CH08}. The stochastic versions of the telegraph equation were intensively studied in the last 50 years, we refer the reader to \cite[Section 1.1]{BG19} for a brief review. The solution theory of the telegraph equation goes back to \cite{CH08}, we present it in the way of \cite[Section 4]{BG19}. In fact, \eqref{eq:telegraph} admits a unique solution which reads  
\begin{align*}
u(X, Y) &= \psi(0) \mathcal{R}(X, Y, 0, 0) + \int_0^{Y} \mathcal{R}(X, Y; 0, y)\big(\psi'(y) + \beta_2 \psi(y)\big) dy + \int_0^X \mathcal{R}(X, Y; x, 0)\big(\chi'(x) + \beta_1 \chi(x)\big) dx\\
\numberthis \label{eq:solutiontelegraph}
&\quad + \int_{0}^{X} \int_0^{Y} \mathcal{R}(X, Y, x,y) f(x, y) dx dy. 
\end{align*}
Here, $\mathcal{R}(X, Y, x, y)$ is the \emph{Riemann function} defined as 
\begin{equation}\label{eq:riemannfunction}
\mathcal{R}(X, Y; x, y) = \frac{1}{2\pi \mathbf{i}} \oint_{-\beta_1} \frac{\beta_2 - \beta_1}{(z+\beta_1) (z+\beta_2)} \exp\bigg[(\beta_1 - \beta_2)\Big(-(X-x) \frac{z}{z + \beta_2} + (Y-y) \frac{z}{z + \beta_1}\Big)\bigg] dz,
\end{equation}
where the contour of the complex integration is a small circle in positive direction which only includes the pole at $-\beta_1$.
When $f$ is given by $f(X, Y) = \sqrt{\theta(X, Y)} \eta(X, Y)$, where $\eta$ is the space-time white noise with dirac delta correlation function and $\theta$ is a deterministic integrable function. By formula \eqref{eq:solutiontelegraph},  the solution to the \emph{stochastic telegraph equation} is a Gaussian field with covariance function 
\begin{align}\label{eq:covariance}
\text{Cov}\big(u(X_1, Y_1), u(X_2, Y_2)\big)  = \int_0^{X_1 \wedge Y_1} \int_0^{X_2 \wedge Y_2} \mathcal{R}(X_1, Y_1, x, y) \mathcal{R}(X_2, Y_2, x, y) \theta(x, y) dx dy.
 \end{align}
\cite[Section 4]{BG19} identifies the following discretization of the telegraph equation 
\begin{equation}\label{eq:discretetelegraph}
\begin{cases}
&\Phi(X+1, Y+1) - b_1 \Phi(X, Y+1) - b_2 \Phi(X+1, Y) + (b_1 + b_2 - 1) \Phi(X, Y) = g(X+1, Y+1),\\
&\Phi(X, 0) = \chi(X), \quad \Phi(0, Y) = \psi(Y),
\end{cases}
\end{equation}
where $\chi(0) = \psi(0)$. The unique solution to \eqref{eq:discretetelegraph} is given by \cite[Theorem 4.7]{BG19}:
\begin{align}
\notag
\Phi(X, Y) &= \psi(0) \mathcal{R}^d (X, Y; 0, 0) + \sum_{y=1}^Y \mathcal{R}^d (X, Y; 0, y) \big(\psi(y) - b_2 \psi(y-1)\big) + \sum_{x = 1}^X \mathcal{R}^d (X, Y; x, 0)\big(\chi(x) - b_1 \chi(x-1)\big)\\ 
\label{eq:temp29}
&\quad + \sum_{x= 1}^X \sum_{y=1}^Y \mathcal{R}^d (X, Y; x, y) g(x, y). 
\end{align}
where the discrete Riemann function $\mathcal{R}^d$ equals (see \cite[Eq. 45]{BG19}) 
\begin{align*}
\mathcal{R}^d (X, Y; x, y) &= \frac{1}{2 \pi \mathsf{i}} \oint_{-\frac{1}{b_2 (1-b_1)}} \frac{(b_2 - b_1) dz}{(1 + b_2 (1-b_1) z) (1+b_1 (1-b_2) z)}\\
\numberthis \label{eq:discreteRiemann}
&\quad \times \Big(\frac{1 + b_1 (1-b_1) z}{1 + b_2 (1-b_1) z}\Big)^{X-x} \Big(\frac{1 + b_2 (1-b_2) z}{1 + b_1 (1-b_2) z}\Big)^{Y-y}.
\end{align*}
Here, the contour is a small circle going in positive direction which only encircles the pole at $-\frac{1}{b_2 (1-b_1)}$. 
\bigskip
\\
In the first version of the arxiv paper \cite{BG18}, Borodin and Gorin showed that under a special scaling regime where the weight of the corner type vertex goes to zero, the height function of the stochastic six vertex model converges to the telegraph equation. They also conjectured that the fluctuation field will converge to the stochastic telegraph equation with some heuristic arguments and proved this result under a special situation called \emph{low density boundary regime}. The result for general boundary condition was later proved in \cite{ST19} and \cite{BG19} via two distinct approaches. This result comes as a surprise. Since from \cite{GS92, BCG16} we know that the stochastic six vertex model belongs to the KPZ universality class. The one point fluctuation of the models in this universality is governed by Tracy Widom distribution \cite{TW94}. However, the solution to the stochastic telegraph equation does not lie in this universality (since it is a Gaussian field). In addition,\cite{CGST20} shows that under weakly asymmetric scaling (which is a different scaling compared with the one in \cite{BG19}), the stochastic six vertex model converges to  the KPZ equation \cite{KPZ86, Cor12}, which is a parabolic stochastic PDE while the stochastic telegraph equation is hyperbolic!
\bigskip
\\
It is natural to ask if the stochastic telegraph equation also arises as a scaling limit of other probabilistic models. In this paper, we show that the stochastic higher spin six vertex (SHS6V) model, which is a higher spin generalization of the stochastic six vertex model, converges to the stochastic telegraph equation under certain scaling regime. This extends the universality of the stochastic telegraph equation. In addition, \cite{Lin20} showed that under a different scaling than the one considered in this paper, the SHS6V model converges to the KPZ equation. This tells us that the SHS6V model converges to two distinct types of stochastic PDE under various choice of scaling.  
\subsection{The SHS6V model}
The SHS6V model  is a four-parameter family of quantum integrable system first introduced in \cite{CP16} and has been intensely studied in recent years, from the perspective of symmetric polynomial \cite{Bor17, Bor18}, exact solvability \cite{BCPS15, CP16, BP18}, Markov duality \cite{CP16, Kuan18, Lin19} and scaling limit \cite{CT17, IMS20, Lin20}. In particular, it is a higher spin generalization of stochastic six vertex model from spin parameter $I = J = 1$ to general $I, J \in \mathbb{Z}_{\geq 1}$. In this paper, we discover a scaling regime for the SHS6V model (which degenerates to the scaling in \cite{BG19} when $I = J = 1$), under which we prove that: 1) the hydrodynamic limit of the SHS6V model is a telegraph equation; 2) the fluctuation field of the model converges to a stochastic telegraph equation. To explain our result with more detail, we start with a brief review of the SHS6V model.
\begin{defin}[$J = 1$ $\mathbb{L}$-matrix] \label{def:j1lmatrix}
We define the $J = 1$ $\mathbb{L}$-matrix to be a matrix with row and column indexed by $\mathbb{Z}_{\geq 0} \times \{0, 1\}$. The element of the $J = 1$ $\mathbb{L}$-matrix is specified by 
\begin{align*}
L^{(1)}_{\alpha}(m, 0; m, 0) &= \frac{1 + \alpha q^m}{1 + \alpha}, \quad L^{(1)}_{\alpha}(m, 0; m-1, 1) = \frac{\alpha (1 - q^m)}{1 + \alpha},\\
L^{(1)}_{\alpha}(m, 1; m, 1) &= \frac{\alpha + \nu q^m}{1 + \alpha}, \quad L^{(1)}_{\alpha}(m, 1; m+1, 0) = \frac{1 - \nu q^m}{1 + \alpha}
\end{align*} 
and $L^{(1)}_\alpha (i_1, j_1; i_2, j_2) = 0$ for all other values of $(i_1, j_1)$, $(i_2, j_2) \in \mathbb{Z}_{\geq 0} \times \{0, 1\}$. As a convention, throughout the paper, we set $\nu = q^{-I}$ for some fixed $I \in \mathbb{Z}_{\geq 1}$. Note that $L_{\alpha}^{(1)} (I, 1; I+1, 0) = 0$, hence the $J = 1$ $\mathbb{L}$-matrix transfers the subspace $\{0,1, \dots, I\} \times \{0, 1\}$ to itself and we will restrict ourselves on this subspace.
\end{defin}
We call $\alpha$ the \emph{spectral parameter} and in the notation of $L_{\alpha}^{(1)}$, where the dependence on other parameters is not made explicit.
It is clear from the definition that for fixed $i_1 \in \{0,1, \dots, I\}$ and $j_1 \in \{0, 1\}$, 
\begin{equation*}
\sum_{(i_2, j_2) \in \{0, 1, \dots, I\} \times \{0, 1\}} L_{\alpha}^{(1)}(i_1, j_1; i_2, j_2) = 1.
\end{equation*}
Moreover, $L_{\alpha}^{(1)}$ is stochastic if we impose the following condition. 
\begin{lemma}\label{lem:stochastic1}
$L_{\alpha}^{(1)}$ is stochastic if one of the following holds:
\begin{itemize}
	\item $q \in (0, 1)$ and $\alpha < - q^{-I}$,
	\item $q > 1$ and $-q^{-I} < \alpha < 0$. 
\end{itemize}
\end{lemma}
\begin{proof}
This follows from \cite[Proposition 2.3]{CP16}, which can also be verified directly.
\end{proof}
For an entry $L^{(1)}_\alpha (i_1, j_1; i_2, j_2)$, we interpret the four tuple $(i_1, j_1, i_2, j_2)$ as a vertex configuration in the sense that a vertex is associated with $i_1$ input lines and $j_1$ input lines coming from bottom and left, $i_2$ output lines and $j_2$ output lines flowing to above and right, see Figure \ref{fig:vertex visualize}. The quantity $L_{\alpha}^{(1)}(i_1, i_2; j_1, j_2)$ gives the weight of the vertex configuration. Note that for a vertex associated with $L_{\alpha}^{(1)}$, we allow up to $I$  number of vertical lines and up to one horizontal line. We say that the $\mathbb{L}$-matrix is conservative in lines as it assigns zero weight to the entry $L_{\alpha}^{(1)}(i_1, j_1; i_2, j_2)$ unless $i_1 + j_1 = i_2 + j_2$. 
\begin{figure}[ht]
\centering
\includegraphics[scale = 0.8]{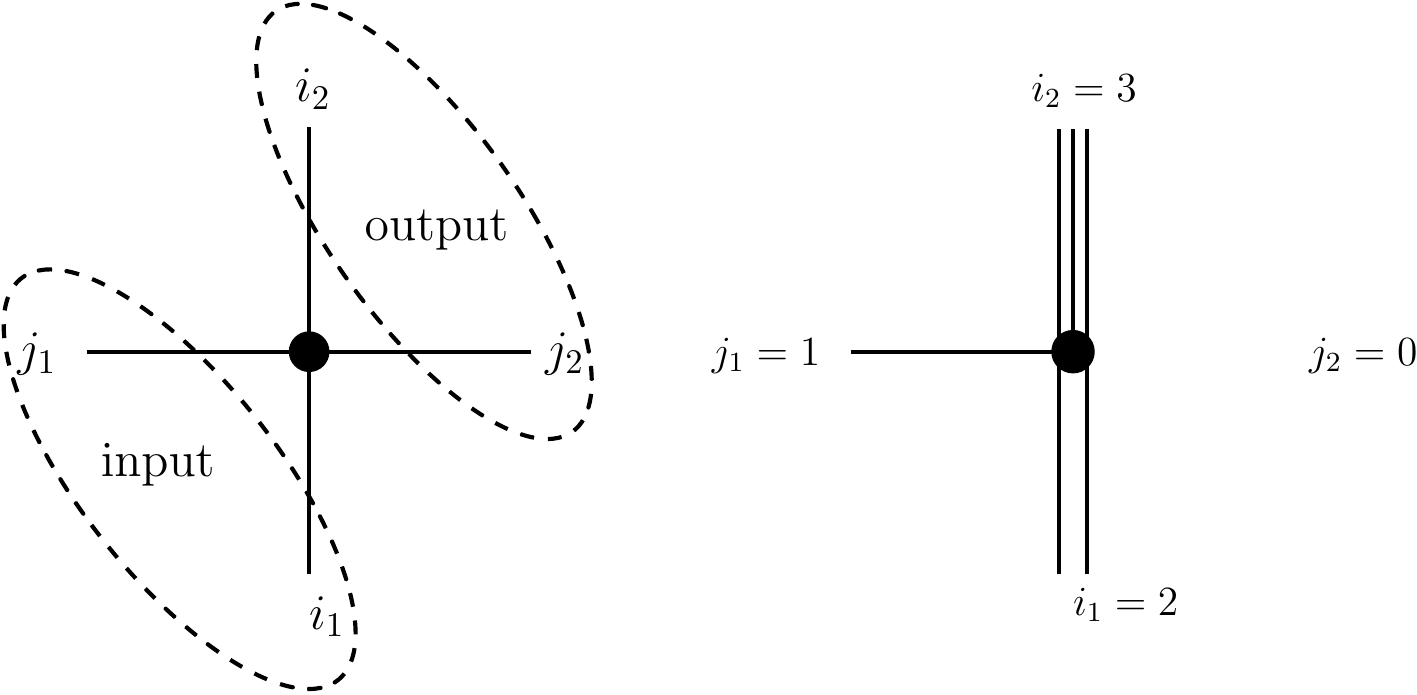}
\caption{Left panel: The vertex configuration labeled by four tuple $(i_1, j_1; i_2, j_2)$ (from bottom and then in the clockwise order) has weight $L_{\alpha}^{(1)}(i_1, j_1; i_2, j_2)$, which absorbs $i_1 \in \{0, 1, \dots, I\}$ input lines from bottom, $j_1 \in \{0, 1\}$ input line from left, and produces $i_2 \in \{0, 1, \dots, I\}$ output lines to above, $j_2 \in \{0, 1\}$ output lines to right. Right panel: Visualization of the vertex configuration $(i_1, j_1; i_2, j_2) = (2, 1; 3, 0)$ in terms of lines.}
\label{fig:vertex visualize}
\end{figure}
\bigskip
\\
We want to relax the restriction that the multiplicities of the horizontal line are bounded by $1$, and instead, consider multiplicities bounded by any fixed $J$. This motivates us to define the $L^{(J)}_\alpha$ matrix, the construction of it follows the so-called fusion procedure, which was invented in a representation-theoretic context \cite{KRS81, KR87}
to produce higher-dimensional solutions of the Yang–Baxter equation from lower-dimensional
ones. The explicit expression of general $J$ $\mathbb{L}$-matrix is derived separately in \cite{Man14} and \cite{CP16}:
\begin{equation}\label{eq:higherspinLmatrix}
\begin{split}
L_\alpha^{(J)}(i_1, j_1; i_2, j_2) = &\mathbf{1}_{\{i_1 + j_1 = i_2 + j_2\}} q^{\frac{2j_1 - j_1^2}{4} - \frac{2j_2 - j_2^2}{4} + \frac{i_2^2 + i_1^2}{4} + \frac{i_2 (j_2 - 1) + i_1 j_1}{2}}\\
&\times \frac{\nu^{j_1 - i_2} \alpha^{j_2 - j_1 + i_2} (-\alpha\nu^{-1}; q)_{j_2 - i_1}}{(q;q)_{i_2} (-\alpha; q)_{i_2 + j_2} (q^{J+1 - j_1 }; q)_{j_1 - j_2}}  { }_{4} \bar{\phi}_3 \bigg(\begin{matrix}
q^{-i_2};q^{-i_1}, -\alpha q^J, -q\nu \alpha^{-1}\\ \nu, q^{1 + j_2 - i_1}, q^{J + 1 - i_2 - j_2} 
\end{matrix} \bigg| q, q\bigg).
\end{split}
\end{equation}
Here, ${}_4 \bar{\phi}_3$ is the regularized terminating basic hyper-geometric series defined by 
\begin{align*}
_{r+1} \bar{\phi}_r \bigg(\begin{matrix}
q^{-n}, a_1, \dots, a_r\\ b, \dots, b_r 
\end{matrix}
\bigg|q, z \bigg) &= \sum_{k=0}^n z^k \frac{(q^{-n}; q)_k}{(q; q)_k} \prod_{i=1}^r (a_i; q)_k (b_i q^k; q)_{n-k}.
\end{align*} 
It is a simple exercise to see when $J = 1$, the expression of $L_{\alpha}^{(J)}$ matches with $L_{\alpha}^{(1)}$ in Definition \ref{def:j1lmatrix}. We will show momentarily that $L_{\alpha}^{(J)}$ is stochastic (Corollary \ref{cor:stochasticj}). This allows us to view the matrix element $L^{(J)}_\alpha (i_1, j_1; i_2, j_2)$ as a vertex configuration in the manner that we described in $J =1$ case. Note that now we allow up to $J$ lines in the horizontal direction.
\bigskip
\\
Despite explicitness, the expression of the $\mathbb{L}$-matrix above is too complicated to manipulate. For instance, using \eqref{eq:higherspinLmatrix} directly, it might be hard to demonstrate the stochasticity of $L_{\alpha}^{(J)}$. To this end, we recall a probabilistic derivation of $L_{\alpha}^{(J)}$ in \cite{CP16} using the idea of fusion, which goes back to \cite{KR87}. We start by introducing a few notations.
\bigskip
\\
Define the stochastic matrix $\Xi$ with rows and columns indexed by  $\{0, 1\}^{\otimes J}$ and $\{0, 1, \dots, J\}$ such that 
\begin{equation*}
\Xi\big((h_1, \dots, h_J), h\big)  = 
\begin{cases}
1 \qquad \text{if } h = \sum_{i=1}^J h_i
\\
0 \qquad \text{else}
\end{cases}
\end{equation*}
and the stochastic matrix $\Lambda$ with row and column indexed by $\{0, 1, \dots, J\}$ and $\{0, 1\}^{\otimes J}$. The matrix element is given by 
\begin{equation*}
\Lambda\big(h, (h_1, \dots, h_J)\big) = 
\begin{cases}
\displaystyle \frac{1}{Z_J(h)}\prod_{i : h_i = 1} q^{i-1} \qquad &\text{if } h = \sum_{i = 1}^J h_i  \\
0 & \text{else}
\end{cases}
\end{equation*}
where $Z_J(h) = q^{h(h-1)/2} \frac{(q, q)_J}{(q, q)_h (q, q)_{J-h}}$ is the normalizing constant (it can be computed using $q$-binomial theorem).
\bigskip
\\
We also define the matrix $L_{\alpha}^{\otimes_q J}$ with rows and columns indexed by $\{0, 1, \dots, I\} \times \{0, 1\}^{\otimes J}$ with matrix elements
$$L_{\alpha}^{\otimes_q J} (v, h_1, \dots, h_J; v', h'_1, \dots, h'_J) = \sum_{\substack{v_0, v_1, \dots, v_J \\ v_0 = v, v_J = v'}} \prod_{i = 1}^J L^{(1)}_{\alpha q^{i-1}} (v_{i-1}, h_i; v_{i}, h'_i).$$
In terms of the right part of Figure \ref{fig:vertexdecompose}, these matrix elements provide the transition probabilities from the lines coming into a column from bottom and left, to those leaving to the top and right. 
\bigskip
\\
The following lemma allows us to decompose the vertex with horizontal spin $J/2$ (i.e. the vertex associated with $L_{\alpha}^{(J)}$) in terms of a sequence of horizontal spin $1/2$ vertices, see Figure \ref{fig:vertexdecompose} for visualization.
\begin{lemma}\label{lem:fusion}
The following identity holds
\begin{align*}
L_{\alpha}^{(J)}(v, h; v', h') = \sum_{\substack{(h_1, \dots, h_J) \in \{0, 1\}^J\\ (h'_1, \dots, h'_J) \in \{0, 1\}^J}} \Lambda\big(h; (h_1, h_2, \dots h_J)\big) L_{\alpha}^{\otimes_q J} (v, h_1, \dots, h_J; v', h'_1, \dots, h'_J)\, \Xi\big((h_1', \dots, h_J'); h'\big).
\end{align*}
\end{lemma}
\begin{proof}
This was shown in \cite[Theorem 3.15]{CP16}.
\end{proof}
\begin{figure}[ht]
\centering 
\includegraphics[scale = 1.3]{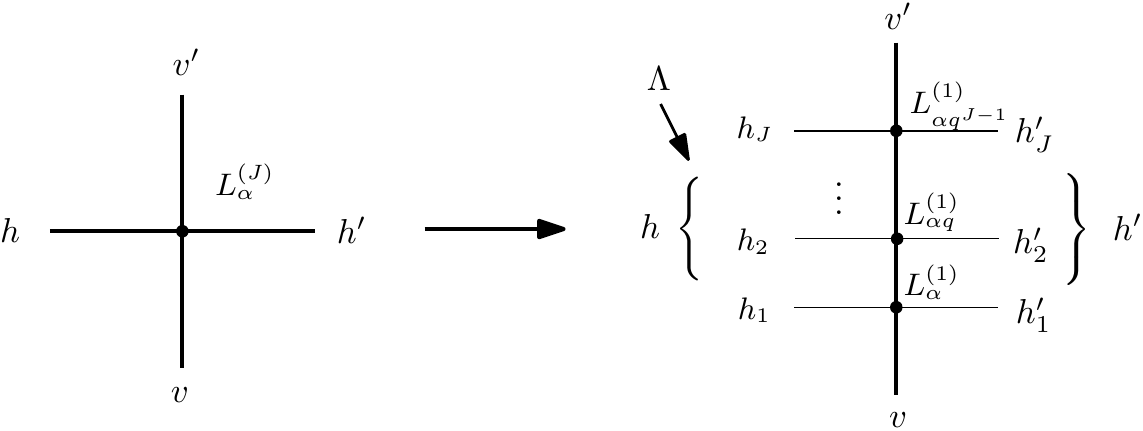}
\caption{Pictorial representation of the identity in Lemma \ref{lem:fusion}. Fixing $h, v, h', v'$, the weight of vertex configuration on the left is given by $L_{\alpha}^{(J)}(\vt, h; \vt', h')$. It is equal to the weight of the column on the right, which is the summation of all $L_{\alpha}^{\otimes_q J} (\vt, h_1, \dots, h_J; \vt', h'_1, \dots, h'_J)$, under the condition $h_1 + \dots + h_J = h$ and $h'_1 + \dots + h_J' = h'$, each term in the summation is reweighted by multiplying $\Lambda\big(h; (h_1, \dots, h_J)\big)$.}	
\label{fig:vertexdecompose}
\end{figure}
Applying Lemma \ref{lem:fusion}, we show that $L_{\alpha}^{(J)}$ is stochastic, under the following choice of parameters.
\begin{cor}\label{cor:stochasticj}
The matrix $L_{\alpha}^{(J)}$ is stochastic if either of the following condition holds
\begin{itemize}
\item $q \in [0, 1)$ and $\alpha < -q^{-I-J+1}$,
\item $q > 1$ and $-q^{-I-J+1} < \alpha < 0$.
\end{itemize}
\end{cor}
\begin{proof}
Note that under the range imposed on $q, \alpha$, referring to Lemma \ref{lem:stochastic1}, the matrix $L^{(1)}_{\alpha q^i}$ is stochastic for each $i  = 0, 1, \dots, J-1$. As the product of stochastic matrices is stochastic as well, the stochasticity of $L_{\alpha}^{(J)}$ follows directly from Lemma \ref{lem:fusion}.  
\end{proof}
We proceed to define the SHS6V model on the first quadrant $\mathbb{Z}_{\geq 0}^2$. For each vertex $(x, y) \in \mathbb{Z}_{\geq 0}^2$, we associate it with a four tuple $(v_{x, y}, h_{x, y}, v_{x, y+1}, h_{x+1, y}) \in \mathbb{Z}_{\geq 0}^4$ such that $v_{x, y}, h_{x, y}$ represent the number of lines entering into the vertex from bottom and left, $v_{x, y+1}, h_{x+1, y}$ denote the number of lines flowing from the vertex to above and right. Note that configurations chosen for two neighboring vertices need to be compatible in the sense that the lines keep flowing. For instance, $v_{x, y+1}$ also represents the number of vertical input lines flowing into $(x, y+1)$, $h_{x, y+1}$ equals the number of horizontal lines entering into $(x+1, y)$ (see the right part of Figure \ref{fig:path}).
\begin{defin}\label{def:SHS6V}
We define the SHS6V model to be a stochastic path ensemble on $\mathbb{Z}_{\geq 0}^2$. The boundary condition specified by $\{v_{x, 0}\}_{x \in \mathbb{Z}_{\geq 0}}$ and $\{h_{0, y}\}_{y \in \mathbb{Z}_{\geq 0}}$   such that $v_{x, 0} \in \{0, 1, \dots, I\}$, $h_{0, y} \in \{0, 1, \dots, J\}$. In other words, we have $h_{0, y}$ number of lines entering into the vertex $(0, y)$ from the left boundary and $v_{x, 0}$ number of  lines flowing into the vertex $(x, 0)$ from the bottom boundary. Sequentially taking $(x, y)$ to be $(0, 0) \to (1, 0) \to (0, 1) \to (2, 0) \to (2, 1) \dots$, for vertex at $(x, y)$, given $v_{x, y}, h_{x, y}$ as the number of vertical and horizontal input lines, we randomly choose the number of vertical and horizontal output lines  $(v_{x, y+1}, h_{x+1, y}) \in \{0, 1, \dots, I\} \times \{0, 1,\dots, J\}$  according to probability  $L_{\alpha}^{(J)}(v_{x, y}, h_{x, y};\, \cdot, \cdot)$. Proceeding with this sequential sampling, we get a collection of paths going to the up-right direction and we call this the SHS6V model.
\end{defin}
We associate a \emph{height function} $H: \mathbb{Z}_{\geq 0}^2 \to \mathbb{Z}$ to the path ensemble, where the paths play a role as the level lines of the height function (see Figure \ref{fig:path}). Define for any $x, y \in \mathbb{Z}_{\geq 0}$, 
\begin{equation*}
H(x, y) = \sum_{j = 1}^{y} h_{0, j-1} - \sum_{i = 1}^{x} v_{i-1, y}.
\end{equation*} 
Clearly, we have $H(0, 0) = 0$ and $H(x, y) - H(x-1, y) = -v_{x-1, y}$. Since the vertex is conservative, we also have
\begin{equation*}
H(x, y) - H(x, y-1) = h_{x, y-1}.
\end{equation*}
Graphically, when we move across $i$ number of vertical lines from left to right, the height function will decrease by $i$. When we  move across $j$ number of horizontal lines, the height function will increase by $j$. We further extend $H(x, y)$ to all $(x, y) \in \mathbb{R}_{\geq 0}^2$ by first linearly interpolating the height function first in the $x$-direction, then in the $y$-direction. It is obvious that the resulting function is Lipschitz and monotone.
\bigskip
\\
For later use, we call $I/2, J/2$ the vertical and horizontal spin respectively. If a vertex is of horizontal spin $1/2$, we call it a $J = 1$ vertex, otherwise we call it a general $J$ vertex. 
\begin{figure}[ht]
\centering
\includegraphics[scale = 1.1]{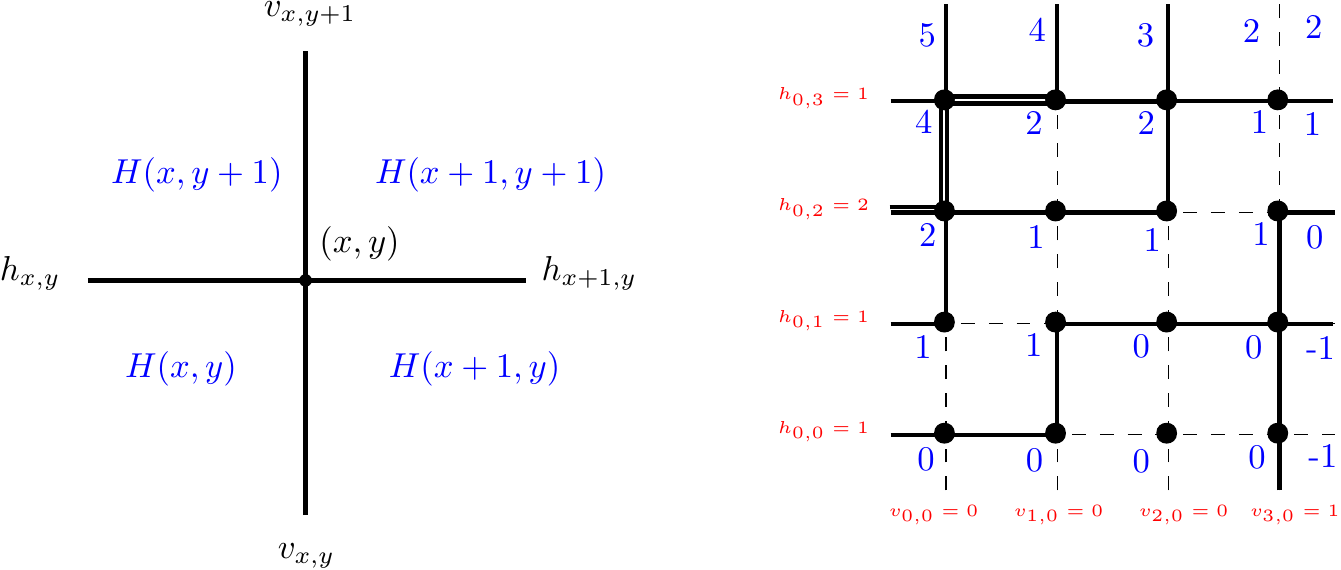}
\caption{Left: Illustration of the height function around a vertex $(x, y)$, note that  $H(x, y+1) = H(x, y) + h_{x, y}$, $H(x+1, y) = H(x, y) - v_{x, y}$ and $H(x+1, y+1) = H(x, y) + h_{x,y} - v_{x, y+1} = H(x, y) - v_{x,y} + h_{x+1, y}$. Right:
Sampled stochastic path ensemble on a quadrant. The red number indicates the number lines entering into the boundary, the blue number represents the height at each vertex.}
\label{fig:path}
\end{figure}
\subsection{Four point relation}\label{subsec:fourpointrelation}
\cite{BG19} shows that the stochastic six vertex model height function converges to a telegraph equation and its fluctuation field converges to a stochastic telegraph equation. The key observation is the following \emph{four point relation}, which says that if we define
\begin{equation*}
\xi^{\text{S6V}}(x+1,y+1) = q^{H(x+1, y+1)} - b_1 q^{H(x, y+1)} - b_2 q^{H(x+1, y)} + (b_1 + b_2  -1) q^{H(x, y)},
\end{equation*}
Here $b_1, b_2$ are the weight of the six vertex model configuration (in our notation $b_1 = \frac{\alpha + \nu}{1 + \alpha}$, $b_2 = \frac{1 + \alpha q}{1 + \alpha}$). Then the conditional expectation and variance of $\xi$ read
\begin{align}\label{eq:introfourpoint}
&\mathbb{E}\Big[\xi^{\text{S6V}}(x+1,y+1) \big| \mathcal{F}(x,y)\Big] = 0,
\\
\label{eq:introfourpointquad}
&\mathbb{E}\Big[\xi^{\text{S6V}}(x+1, y+1)^2 \big| \mathcal{F}(x,y)\Big] = \gamma_1 \Delta_x \Delta_y + \gamma_2 q^{H(x, y)} \Delta_x  + \gamma_3 q^{H(x, y)} \Delta_y,
\end{align}
where $\mathcal{F}(x,y)$ is a sigma algebra generated by $\{H(u, v): u \leq x \text{ or } v \leq y\}$ and $\Delta_x := q^{H(x+1, y)} - q^{H(x, y)}$, $\Delta_y := q^{H(x, y+1)} - q^{H(x, y)}$.  The parameters $\gamma_i, i = 1,2,3 $ depend on $b_1, b_2$. 
\bigskip
\\
In our paper, we generalize the above relations to the SHS6V model. Define 
\begin{equation*}
\xi^{\text{S6SHV}} (x+1, y+1) = q^{H(x+1, y+1)} - \frac{\alpha + \nu}{1 + \alpha} q^{H(x, y+1)} - \frac{1 + \alpha q^J}{1 + \alpha} q^{H(x+1, y)} + \frac{\nu + \alpha q^J}{1 + \alpha} q^{H(x, y)},
\end{equation*}
We prove (respectively in Theorem \ref{thm:fourpoint} and Theorem \ref{thm:fourpointquad}) that  
\begin{align}\label{eq:introfourpoint1}
&\mathbb{E}\Big[\xi^{\text{SHS6V}}(x+1,y+1) \big| \mathcal{F}(x,y)\Big] = 0,\\
\label{eq:introfourpointquad1}
&\mathbb{E}\Big[\xi^{\text{SHS6V}}(x+1, y+1)^2 \big| \mathcal{F}(x,y)\Big] = \gamma_1 \Delta_x \Delta_y + \gamma_2 q^{H(x, y)} \Delta_x + \gamma_3 q^{H(x, y)} \Delta_y + \mathbf{R}(x, y).
\end{align}
$\mathbf{R}(x, y)$ is an error term that is negligible under our scaling. From now on, we may also use $\xi$ to denote $\xi^{\text{SHS6V}}$.
\bigskip
\\
Why does such a generalization exist? In the context of the stochastic six vertex model, \eqref{eq:introfourpoint} is related to the self-duality discovered in \cite[Proposition 2.20]{CP16}, though it is more of a local relation than the way duality is generally stated (it is unclear to us how to prove \eqref{eq:introfourpoint} from the duality directly). In fact, \cite[Corollary 3.3]{CP16} shows that the SHS6V model with general $I, J$ enjoys the same self-duality, so it is natural to expect that \eqref{eq:introfourpoint1}, as a generalized version of \eqref{eq:introfourpoint} holds.  For the quadratic variation, the situation is more subtle for the SHS6V model. We do not come up with a simple reason why \eqref{eq:introfourpointquad1} holds, though this may be understandable from our proof, which is briefly explained in the next paragraph. Here, we just emphasize that as shown in Remark \ref{rmk:quadratic}, there exist no $\gamma_i, i = 1,2,3 $ such that the identity without an error term holds for the SHS6V model. We also emphasize that it is only under our scaling \eqref{eq:scaling} that $\mathbf{R}(x, y)$ is negligible.
\bigskip
\\
Let us explain the ideas and techniques used in proving \eqref{eq:introfourpoint1} and \eqref{eq:introfourpointquad1}. In \cite{BG19}, the authors prove \eqref{eq:introfourpoint} and \eqref{eq:introfourpointquad} via a direct computation, which corresponds to enumerating all possible six vertex configurations. In our case, the situation is more involved: when $J$ is large, the expression of $L_{\alpha}^{(J)}$ is so complicated that it is hopeless to check these relations directly. Alternatively, we first verify them directly for $J = 1$, in which case the $\mathbb{L}$-matrix has a simple expression given by Definition \ref{def:j1lmatrix}. For general $J$, we use fusion, which allows us to decompose the general $J$ vertex into a sequence of $J = 1$ vertices (see Figure \ref{fig:vertexdecompose}). Repeatedly using the $J = 1$ version of \eqref{eq:introfourpoint1} (where the spectral parameter $\alpha$ is replaced by $\alpha q^i$ in the expression of $\xi$), we get $J$ identities. Summing up these identities in a clever way, we see a telescoping property and \eqref{eq:introfourpoint1} follows. To prove \eqref{eq:introfourpointquad1}, besides using fusion, we need to refer to the property of our scaling \eqref{eq:scaling}, which says that
with a probability converging to $1$, the lines entering into a vertex will keep flowing in the same direction (see Lemma \ref{lem:scalingmatrix}).
\bigskip
\\  
In \cite{CP16}, the fusion was stated in a way that the spectral parameters progress geometrically by $q$ from bottom to top when we decompose the general $J$ vertex to a column of $J = 1$ vertices. It turns out that (Lemma \ref{lem:vertexid}) we can also reverse the direction and let the parameters progress geometrically by $q$ from top to bottom (meanwhile we change the probability distribution assigned on the input lines from the left). We did not see this result elsewhere. Note that it is only after this reversal of the spectral parameters that we obtain the telescoping property mentioned in the previous paragraph.
\subsection{Stochastic telegraph equation as a scaling limit of the SHS6V model}
Having established the four point relation, we are ready to talk about our result. We show that under our scaling,
\begin{enumerate}[(i).]
\item (Hydrodynamic limit (or law of large numbers) -- Theorem \ref{thm:lln}): The SHS6V model height function converges uniformly in probability to a telegraph equation.
\item (Functional central limit theorem -- Theorem \ref{thm:clt} (also see Corollary \ref{cor:clt})): The fluctuation field of the height function around its hydrodynamic limit (viewed as a random continuous function) converges weakly to a stochastic telegraph equation.
\end{enumerate}
Once we have proved the four point relation for the SHS6V model, the proof for the law of large numbers is akin to \cite[Theorem 5.1]{BG19}. For the functional central limit theorem, 
our proof breaks down into proving the finite dimensional weak convergence (Proposition \ref{prop:cltfinitedim}) and tightness (Proposition \ref{prop:tightness}). For finite dimensional convergence,
the proof follows a similar idea as in \cite[Theorem 6.1]{BG19}, subject to certain generalization. For 
the tightness, we rely on the Burkholder inequality and a careful control of joint moments of $\xi$ at different locations (Lemma \ref{lem:technical bound}). We remark that the proof of the tightness may not fit to the regime of classical functional martingale CLT result (e.g. \cite[Section 6]{Bro71}), see Remark \ref{rmk:tightness} for more discussion.
\bigskip
\\ 
To present our results, let us first introduce our scaling.
Fix $I, J \in \mathbb{Z}_{\geq 1}$ and positive $\beta_1, \beta_2$ such that $\beta_1 \neq \beta_2$ , we scale the parameter $q, \alpha$ in the way that 
\begin{equation}\label{eq:scaling}
q =e^\frac{\beta_1 - \beta_2}{L}, \qquad \frac{1 + \alpha q^J}{1 + \alpha} = 
e^{-\frac{J\beta_2}{L}}, \quad L \to \infty.
\end{equation}
It is straightforward that 
as $L \to \infty$, $\alpha$ and $q$ always satisfy one of the conditions given in Corollary  \ref{cor:stochasticj}, thus $L_{\alpha}^{(J)}$ is indeed stochastic.
\begin{theorem}\label{thm:lln}
Define $\mathfrak{q} = e^{\beta_1 - \beta_2}$ and fix $A, B > 0$, consider two monotone Lipschitz functions $\chi$ and $\psi$. 
Suppose that the boundary for the SHS6V model is chosen in the way that as $L \to \infty$,
$\frac{1}{L} H(Lx, 0)  \to \chi(x)$ and $\frac{1}{L} H(0, Ly) \to \psi(y)$ uniformly in probability for $x \in [0, A]$ and $y \in [0, B]$, then as $L \to \infty$,
\begin{equation*}
 \frac{1}{L} \sup_{x \in [0, A] \times [0, B]} |H(Lx, Ly) - L \mathbf{h}(x, y)| \p 0,
\end{equation*}
where $\p$ means the convergence in probability. $\mathfrak{q}^{\mathbf{h}(x, y)}$ is the unique solution to the telegraph equation
\begin{equation}\label{eq:llntelegraph}
\frac{\partial^2}{\partial x \partial y} \mathfrak{q}^{\mathbf{h}(x, y)} + J\beta_2 \frac{\partial}{\partial x} \mathfrak{q}^{\mathbf{h}(x, y)} + I\beta_1 \frac{\partial}{\partial y} \mathfrak{q}^{\mathbf{h}(x, y)} = 0,
\end{equation}
with the boundary condition specified by $\mathfrak{q}^{\mathbf{h}(x, 0)} = \mathfrak{q}^{\chi(x)}$ and $\mathfrak{q}^{\mathbf{h(0, y)}} = \mathfrak{q}^{\psi(y)}$. 
\end{theorem}
We remark that there is a typo in \cite[Eq. 69]{BG19} about the boundary condition, $\mathfrak{q}^{\mathbf{h}(x, 0)}, \mathfrak{q}^{\mathbf{h}(0, y)}$ should equal $\mathfrak{q}^{\chi(x)}$ and $\mathfrak{q}^{\psi(y)}$, instead of $\chi(x)$ and $\psi(y)$. 
\bigskip
\\
Having established the law of large number for the height function, we proceed to show the functional central limit theorem. As a convention, we endow the space $C(\mathbb{R}^2_{\geq 0})$ with the topology of uniform convergence over compact subsets and use $``\Rightarrow"$ to denote the weak convergence. Recall that we linearly extend $H(x, y)$ for non-integer $x, y$, so $H(x, y) \in C(\mathbb{R}_{\geq 0}^2)$.
\begin{theorem}\label{thm:clt}
Assuming further that $\chi(x)$ and $\psi(y)$ are piecewise $C^1$-smooth, we have the weak convergence as $L \to \infty$,
\begin{equation*}
\sqrt{L} \bigg(q^{H(Lx, Ly)} - \mathbb{E}\Big[q^{H(Lx, Ly)}\Big]\bigg) \Rightarrow \varphi(x, y) \quad \text{in } C(\mathbb{R}^2_{\geq 0}),
\end{equation*} 
where $\varphi(x, y)$ is a random continuous function which solves the stochastic telegraph equation
\begin{equation}\label{eq:temp13}
\varphi_{xy} + I \beta_1 \varphi_y + J \beta_2 \varphi_x = \eta \cdot \sqrt{(\beta_1 + \beta_2) \mathfrak{q}^{\mathbf{h}}_x \mathfrak{q}^{\mathbf{h}}_y + J(\beta_2 - \beta_1) \beta_2 \mathfrak{q}^{\mathbf{h}} \mathfrak{q}^{\mathbf{h}}_x + I(\beta_1 - \beta_2) \beta_1 \mathfrak{q}^{\mathbf{h}} \mathfrak{q}^{\mathbf{h}}_y},
\end{equation}
Here, $\mathfrak{q}_x^{\mathbf{h}} := \partial_x (\mathfrak{q}^{\mathbf{h}(x, y)})$ and  $\mathfrak{q}_y^{\mathbf{h}} := \partial_y (\mathfrak{q}^{\mathbf{h}(x, y)})$, the boundary of $\varphi$ is given by zero. 
\end{theorem}
\begin{remark}\label{rmk:covariance}
By \eqref{eq:covariance}, it is clear  that $\varphi$ is a Gaussian field with covariance function 
\begin{align*}
\emph{Cov}\Big(\varphi(X_1, Y_1), \varphi(X_2, Y_2)\Big) =& \!\int_0^{X_1 \wedge Y_1}\! \int_0^{X_2 \wedge Y_2}\! \mathcal{R}_{IJ}(X_1, Y_1, x, y) \mathcal{R}_{IJ}(X_2, Y_2, x, y)\\ &\times \Big((\beta_1 + \beta_2) \mathfrak{q}^{\mathbf{h}}_x \mathfrak{q}^{\mathbf{h}}_y + J(\beta_2 - \beta_1) \beta_2 \mathfrak{q}^{\mathbf{h}} \mathfrak{q}^{\mathbf{h}}_x + I(\beta_1 - \beta_2) \beta_1 \mathfrak{q}^{\mathbf{h}} \mathfrak{q}^{\mathbf{h}}_y \Big) dxdy,
\end{align*}
where $\mathcal{R}_{IJ}$ is the Riemann function in \eqref{eq:riemannfunction} with $\beta_1$ and $\beta_2$ replaced by $I \beta_1$ and $J \beta_2$ respectively, i.e.
\begin{equation}\label{eq:riemannIJ}
\mathcal{R}_{IJ} (X, Y; x, y) = \frac{1}{2\pi \mathbf{i}} \oint_{-I\beta_1} \frac{J\beta_2 - I\beta_1}{(z+ I \beta_1) (z+J\beta_2)} \exp\bigg[(I\beta_1 - J\beta_2)\Big(-(X-x) \frac{z}{z + J\beta_2} + (Y-y) \frac{z}{z + I\beta_1}\Big)\bigg] dz,
\end{equation}
\end{remark}
As a corollary of the previous results, we have the following.
\begin{cor}\label{cor:clt}
As $L \to \infty$,
\begin{equation*}
\frac{H(Lx, Ly) - \mathbb{E}\big[H(Lx, Ly)\big]}{\sqrt{L}} \Rightarrow \phi(x, y) \quad \text{in } C(\mathbb{R}^2_{\geq 0}),
\end{equation*}
$\phi(x, y)$ is a Gaussian field given by $\phi(x, y):= \frac{\varphi(x, y)}{\mathfrak{q}^{\mathbf{h}(x, y)} \log \mathfrak{q}} $, which solves 
\begin{equation}\label{eq:temp11}
\phi_{xy} + I\beta_1 \phi_y + J\beta_2 \phi_x + (\beta_1 - \beta_2) (\phi_y \mathbf{h}_x + \phi_x \mathbf{h}_y) = \eta \cdot \sqrt{(\beta_1 + \beta_2) \mathbf{h}_x \mathbf{h}_y - J\beta_2 \mathbf{h}_x + I\beta_1 \mathbf{h}_y }.
\end{equation}
\end{cor}
The rest of the paper is organized as follows. In Section \ref{sec:fourpoint}, we first establish an identity (Lemma \ref{lem:vertexid}), which gives an alternative way to apply fusion. Then, we prove our four point relation (Theorem \ref{thm:fourpoint} and Theorem \ref{thm:fourpointquad}). We also discuss some properties of our scaling (Lemma \ref{lem:scalingmatrix}). In Section \ref{sec:proofofresults}, we first use the four point relation to prove the law of large numbers (Theorem \ref{thm:lln}) and the finite dimensional version of the CLT (Proposition \ref{prop:cltfinitedim}). Then we establish the tightness (Proposition \ref{prop:tightness}) and improve our CLT to the functional level (Theorem \ref{thm:clt}).
\subsection{Acknowledgment}
The author wants to thank Ivan Corwin for many valuable comments on the paper; Vadim Gorin for helpful comments and discussion; and Shalin Parekh for an inspiring discussion about the tightness result.  The author was supported by Ivan Corwin through the NSF grants DMS-1811143, DMS-1664650
and also by the Minerva Foundation Summer Fellowship program.
\section{Four point relation}
\label{sec:fourpoint}
In this section, we prove the four point relation \eqref{eq:introfourpoint1} and \eqref{eq:introfourpointquad1} that mentioned in Section \ref{subsec:fourpointrelation}. To begin with, we present a lemma that allows us to reverse the spectral parameters upside down when we decompose the general $J$ vertex into a column of $J=1$ vertices, see Figure \ref{fig:vertexreverse1} for visualization. The key for our proof is an identity that allows us to switch a pair of vertices with different spectral parameters, see Figure \ref{fig:vertexreverse2}. We do not find such identity in the literature. It seems to us that this identity does not follow directly from the Yang-Baxter equation.  
\bigskip
\\
Define the stochastic matrix $\widetilde{\Lambda}$,
\begin{equation*}
\widetilde{\Lambda}(h, (h_1, \dots, h_J)) := 
\begin{cases}
\displaystyle \frac{1}{Z_J(h)}\prod_{h_i = 1} q^{J-i}  \qquad  &\text{ if } h = \sum_{i=1}^J h_i\\
0 \hspace{6.5em} &\text{ else}
\end{cases}
\end{equation*}
and 
\begin{equation*}
\widetilde{L}_{\alpha}^{\otimes_q J} (v, h_1, \dots, h_J; v', h'_1, \dots, h'_J) := \sum_{\substack{v_0, v_1, \dots, v_J \\ v_0 = v, v_J = v'}}\prod_{i = 1}^J L^{(1)}_{\alpha q^{J-i}} (v_{i-1}, h_i; v_{i}, h'_i).
\end{equation*}
Note that comparing with the expression of $\Lambda$ and $L^{\otimes_q J}_{\alpha}$, the term $q^{i-1}$ is replaced by $q^{J-i}$, which corresponds to reversing the spectral parameters upside down.
\begin{lemma}\label{lem:vertexid}
For fixed $h, v, h', v'$, the following identity holds,
\begin{align*}
&\sum_{\substack{(h_1, \dots, h_J) \in \{0, 1\}^J\\ (h'_1, \dots, h'_J) \in \{0, 1\}^J}} \Lambda(h; h_1, h_2, \dots h_J) L_{\alpha}^{\otimes_q J} (v, h_1, \dots, h_J; v', h'_1, \dots, h'_J)\, \Xi(h_1', \dots, h_J'; h')\\ 
\numberthis \label{eq:vertexreverse}
&= \sum_{\substack{(h_1, \dots, h_J) \in \{0, 1\}^J\\ (h'_1, \dots, h'_J) \in \{0, 1\}^J}} \widetilde{\Lambda}(h; h_1, h_2, \dots h_J) \widetilde{L}_{\alpha}^{\otimes_q J} (v, h_1, \dots, h_J; v', h'_1, \dots, h'_J)\, \Xi(h_1', \dots, h_J'; h').
\end{align*}
Consequently, we have alternate expression for the general $J$ vertex weight 
\begin{align}\label{eq:fusionreverse}
L_{\alpha}^{(J)}(v, h; v', h') = \sum_{\substack{(h_1, \dots, h_J) \in \{0, 1\}^J\\ (h'_1, \dots, h'_J) \in \{0, 1\}^J}} \widetilde{\Lambda}(h; h_1, h_2, \dots h_J) \widetilde{L}_{\alpha}^{\otimes_q J} (v, h_1, \dots, h_J; v', h'_1, \dots, h'_J)\, \Xi(h_1', \dots, h_J'; h').
\end{align}
\end{lemma} 
\begin{proof}
By Lemma \ref{lem:fusion}, it is clear that \eqref{eq:vertexreverse} implies \eqref{eq:fusionreverse}. It suffices to prove \eqref{eq:vertexreverse}, which says, graphically 
\begin{figure}[ht]
\centering 
\includegraphics[scale = 1]{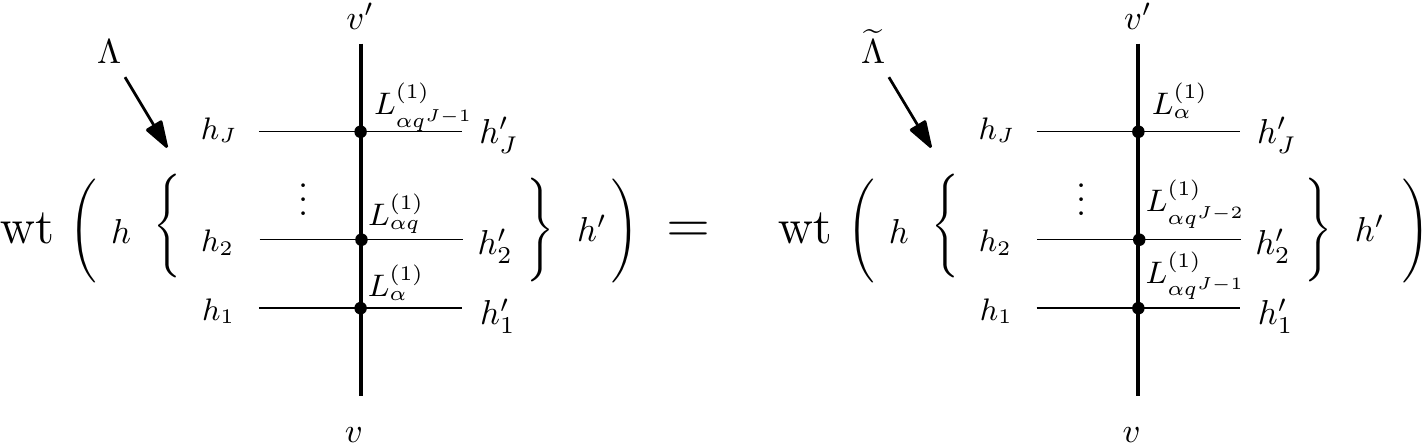}
\caption{Pictorial representation of the identity \eqref{eq:vertexreverse}. The weight (wt) of a
diagram is given by a summation of products of $\mathbb{L}$-matrices over $h_1, \dots, h_J$, with condition $h_1 + \dots + h_J = h$ and $h'_1 + \dots + h'_J = h'$. Each product on the left (resp. right) hand side in the summation is reweighted by $\Lambda(h; h_1, \dots, h_J)$ (resp. $\widetilde{\Lambda}(h; h_1, \dots, h_J)$).}
\label{fig:vertexreverse1}
\end{figure}
\bigskip
\\
When $J = 1$, the proof is trivial. When $J =2$, the identity \eqref{eq:vertexreverse} reduces to Figure \ref{fig:vertexreverse2}.
\begin{figure}[ht]
\centering
\includegraphics[scale = 1]{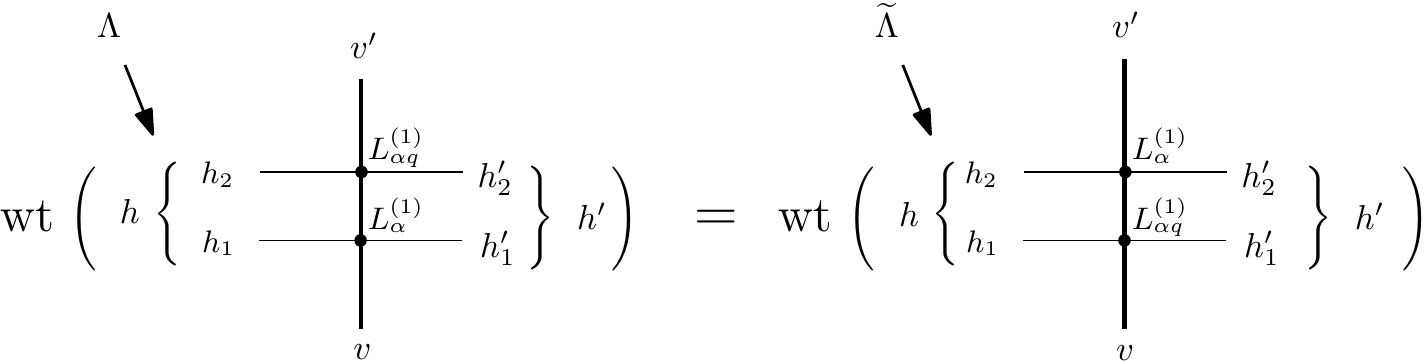}
\caption{Identity \eqref{eq:vertexreverse} when $J = 2$.}
\label{fig:vertexreverse2}
\end{figure}
Since $h, h' \in \{0, 1, 2\}$, there are nine cases in total. One can verify each case directly and here, we  
only show our verification for $h = 1$ and $h' = 1$, in which case the computation is more involved. The LHS in Figure \ref{fig:vertexreverse2} equals   
\begin{align*}
&\Lambda\big(1; (1, 0)\big) \Big(L^{(1)}_{\alpha}(\vt, 1; \vt, 1) L^{(1)}_{\alpha q}(\vt, 0; \vt, 0) + L^{(1)}_{\alpha}(\vt, 1; \vt+1, 0) L^{(1)}_{\alpha q}(\vt+1, 0; \vt, 1) \Big) \\ 
&+ \Lambda\big(1; (0, 1)\big) \Big(L^{(1)}_{\alpha}(\vt, 0; \vt-1, 1) L^{(1)}_{\alpha q}(\vt-1, 1; \vt, 0) + L^{(1)}_{\alpha}(\vt, 0; \vt, 0) L^{(1)}_{\alpha q}(\vt, 1; \vt, 1)\Big) \\
\numberthis\label{eq:temp25}
&= \frac{1}{1 + q} \Big(\frac{\alpha + \nu q^{\vt}}{1 + \alpha} \frac{1 + \alpha q^{\vt+1}}{1 + \alpha q} + \frac{1 - \nu q^\vt}{1 + \alpha} \frac{\alpha q (1-  q^{\vt + 1})}{1 + \alpha q}\Big) + \frac{q}{1 + q} \Big(\frac{\alpha(1 - q^{\vt})}{1 + \alpha} \frac{1 - \nu q^{\vt-1}}{1 + \alpha q} + \frac{1 + \alpha q^\vt}{1 + \alpha} \frac{\alpha q + \nu q^{\vt}}{1 + \alpha q}\Big)
\end{align*}
and the RHS equals 
\begin{align*}
&\widetilde{\Lambda}\big(1; (1, 0)\big) \Big(L^{(1)}_{\alpha q}(\vt, 1; \vt, 1) L^{(1)}_{\alpha}(\vt, 0; \vt, 0) + L^{(1)}_{\alpha q}(\vt, 1; \vt+1, 0) L^{(1)}_{\alpha }(\vt+1, 0; \vt, 1) \Big) \\ 
&+ \widetilde{\Lambda}\big(1; (0, 1)\big) \Big(L^{(1)}_{\alpha q}(\vt, 0; \vt-1, 1) L^{(1)}_{\alpha}(\vt-1, 1; \vt, 0) + L^{(1)}_{\alpha q}(\vt, 0; \vt, 0) L^{(1)}_{\alpha}(\vt, 1; \vt, 1)\Big) \\
\numberthis\label{eq:temp26}
&= \frac{q}{1 + q} \Big(\frac{\alpha q + \nu q^{\vt}}{1 + \alpha q} \frac{1 + \alpha q^{\vt}}{1 + \alpha } + \frac{1 - \nu q^\vt}{1 + \alpha q} \frac{\alpha  (1-  q^{\vt + 1})}{1 + \alpha }\Big) + \frac{1}{1 + q} \Big(\frac{\alpha q (1 - q^{\vt})}{1 + \alpha q} \frac{1 - \nu q^{\vt-1}}{1 + \alpha} + \frac{1 + \alpha q^{\vt+1}}{1 + \alpha q} \frac{\alpha + \nu q^{\vt}}{1 + \alpha}\Big)
\end{align*}
It is not hard to see directly that the RHS of \eqref{eq:temp25} and \eqref{eq:temp26} are both the sum of the following four terms (divided by a common denominator $(1+q)(1+\alpha) (1+\alpha q)$)
$$q(\alpha q + \nu q^{\vt})(1 + \alpha q^{\vt}),\qquad q\alpha(1-\nu q^v)(1-q^{v+1}),\quad \alpha q(1-q^\vt)(1-\nu q^{\vt-1}), \quad (1+\alpha q^{v+1})(\alpha + \nu q^\vt). $$
For the verification of other $h, h' \in \{0, 1, 2\}$, we omit the details of our computation.
\bigskip
\\
For general $J$, we look at the column of vertices on the  LHS of the equation illustrated in Figure \ref{fig:vertexreverse1}. From bottom to top, we label the vertices from $1$ to $J$.  Sequentially for $i = 1, \dots, J-1$, we apply the $J = 2$ identity (that we just verified) for the vertex $i$ and $i+1$ in that column.  Then, the spectral parameters of the vertices (looking from bottom to top) change from  $(\alpha, \alpha q, \dots, \alpha q^{J-1})$ to $(\alpha q, \alpha q^2, \dots, \alpha q^{J-1}, \alpha)$, note that the vertex with spectral parameter $\alpha$ moves from bottom to top. The $\Lambda$ also changes accordingly. Then we apply the $J=2$ identity for $i =1, \dots, J-2$ to move the spectral parameter $\alpha q$ to the second top place.  If we keep implementing this procedure, finally we get a column of vertices with spectral parameters $(\alpha q^{J-1}, \alpha q^{J-2}, \dots, \alpha)$. The left input lines are weighted by $\widetilde{\Lambda}$.
\end{proof}
\begin{remark}
It turns out that following the same argument, the identities \eqref{eq:vertexreverse}, \eqref{eq:fusionreverse} also hold when we replace the stochastic matrix $\widetilde{\Lambda}$ with 
\begin{equation*}
\Lambda_\sigma\big(h, (h_1, \dots, h_J)\big) := 
\begin{cases}
\displaystyle \frac{1}{{Z}_J (h)}\prod_{h_i = 1} q^{\sigma(i) - 1}  \qquad  &\emph{ if } h = \sum_{i=1}^J h_i,\\
0 \hspace{7.2em} &\emph{ else,}
\end{cases}
\end{equation*}
and replace $\widetilde{L}_{\alpha}^{\otimes_q J} (v, h_1, \dots, h_J; v', h'_1, \dots, h'_J)$ with
\begin{equation*}
L_{\sigma, \alpha}^{\otimes_q J} (v, h_1, \dots, h_J; v', h'_1, \dots, h'_J) :=  \sum_{\substack{v_0, v_1, \dots, v_J \\ v_0 = v, v_J = v'}}\prod_{i = 1}^J L^{(1)}_{\alpha q^{\sigma(i) - 1}} (v_{i-1}, h_i; v_{i}, h'_i).
\end{equation*}
where $\sigma$ is an arbitrary permutation of $\{1, 2, \dots, J\}$. We do not include this generalization in the lemma since we are not going to use it.
\end{remark}
\begin{theorem}\label{thm:fourpoint}
Consider the SHS6V model associated with the height function $H$, define for $x, y \in \mathbb{Z}_{\geq 0}$,
\begin{equation}\label{eq:xi}
\xi(x+1, y+1) = q^{H(x+1, y+1)} -  \frac{\alpha + \nu}{1 + \alpha} q^{H(x, y+1)} - \frac{1+\alpha q^J}{1 + \alpha} q^{H(x+1, y)} + \frac{\nu + \alpha q^J}{1 + \alpha} q^{H(x, y)},
\end{equation}
then we have, 
\begin{equation}\label{eq:fourpoint1}
\mathbb{E}\Big[\xi(x+1, y+1) \big| \mathcal{F}(x,y)\Big] = 0,
\end{equation}
where $\mathcal{F}(x,y) = \sigma\big(H(i, j): i\leq x \text{ or } j \leq y\big)$.
\end{theorem}
\begin{proof}
Since our model is homogeneous, i.e. every vertex is assigned with the same $\mathbb{L}$-matrix, we suppress the dependence on $x, y$ in our notation and denote by 
\begin{equation*}
\xi:= \xi(x+1, y+1), \qquad \HH := H(x, y),\qquad h := H(x, y+1) - H(x, y), \qquad  \vt := H(x, y) - H(x+1, y).
\end{equation*}  
In addition, we let
\begin{equation*}
\FF:= \sigma\big(H(x, y), H(x, y+1), H(x+1, y)\big) = \sigma\big(\mathsf{H}, h, \vt\big).
\end{equation*}
By the sequential update rule specified in Definition \ref{def:SHS6V}, $H(x+1, y+1)$ only depends on the information of $\mathsf{H}, h, \vt$, so 
\begin{equation*}
\mathbb{E}\Big[\xi \big| \mathcal{F}(x,y) \Big] = \mathbb{E}\Big[\xi \big| \mathcal{F}\Big].
\end{equation*}
To prove \eqref{eq:fourpoint1}, it suffices to show that 
\begin{align}\label{eq:fourpoint}
\mathbb{E}\Big[\xi \big| \mathcal{F}\Big] = 0.
\end{align}
We prove this identity in two steps:
\bigskip
\\
\textbf{Step 1} ($J=1$):
We assume $J = 1$, in which case the vertex weight \eqref{eq:higherspinLmatrix} reduces to the weights in Definition \ref{def:j1lmatrix}. Let us verify \eqref{eq:fourpoint} directly,
\begin{align*}
\mathbb{E}\Big[\xi \big| \mathcal{F}\Big] &= \mathbb{E}\Big[q^{H(x+1, y+1)} -  \frac{\alpha + \nu}{1 + \alpha} q^{H(x, y+1)} - \frac{1+\alpha q}{1 + \alpha} q^{H(x+1, y)} + \frac{\nu + \alpha q}{1 + \alpha} q^{H(x, y)} \big| \mathcal{F}\Big],
\\
&= \mathbb{E}\Big[q^{H(x+1, y+1)} \big| \mathcal{F}\Big] - \frac{\alpha + \nu}{1 + \alpha} q^{H(x, y+1)} - \frac{1+\alpha q}{1 + \alpha} q^{H(x+1, y)} + \frac{\nu + \alpha q}{1 + \alpha} q^{H(x, y)},\\
&= \mathbb{E}\Big[q^{H(x+1, y+1)}\big| \mathcal{F}\Big] - \frac{\alpha + \nu}{1 + \alpha} q^{\HH + h} - \frac{1 + \alpha q}{1 + \alpha} q^{\HH - \vt} + \frac{\nu+\alpha q}{1+ \alpha} q^{\HH}.
\end{align*}
Since $J =1 $, $h$ is either $0$ or $1$, we discuss them respectively.
\bigskip
\\
If $h = 0$, i.e. $H(x, y+1) = \HH$, by Definition \ref{def:j1lmatrix},
\begin{equation}\label{eq:distributeone}
\mathbb{P}\Big(H(x+1, y+1) = \HH-\vt\Big) = \frac{1 + \alpha q^\vt}{1 + \alpha}; \quad \mathbb{P}\Big(H(x+1, y+1) = \HH- \vt + 1\Big) = \frac{\alpha(1-q^\vt)}{1 + \alpha}.
\end{equation}
Hence,
\begin{align*}
\mathbb{E}\Big[\xi \big| \mathcal{F}\Big] = \frac{1 + \alpha q^\vt}{1 + \alpha} q^{\HH-\vt} + \frac{\alpha(1-q^\vt)}{1+\alpha} q^{\HH-\vt+1} - \frac{\alpha+\nu}{1+\alpha} q^\HH- \frac{1+\alpha q}{1+\alpha } q^{\HH-\vt} + \frac{\nu + \alpha q}{1+\alpha} q^\HH= 0.
\end{align*}
If $h = 1$, i.e. $H(x, y+1) = \HH+ 1$, we have
\begin{equation}\label{eq:distributetwo}
\mathbb{P}\Big(H(x+1, y+1) = \HH-\vt \Big) = \frac{1 - \nu q^\vt}{1 + \alpha}; \quad \mathbb{P}\Big(H(x+1, y+1) = \HH-\vt + 1 \Big) = \frac{\alpha + \nu q^\vt}{1 + \alpha},
\end{equation}
which yields
\begin{align*}
&\mathbb{E}\Big[\xi \big| \mathcal{F}\Big] = 
\frac{1-\nu q^\vt}{1 + \alpha} q^{\HH-\vt} + \frac{\alpha + \nu q^\vt}{1+\alpha} q^{\HH-\vt+1} - \frac{\alpha + \nu}{1 + \alpha} q^{\HH+1} - \frac{1 + \alpha q}{1 + \alpha} q^{\HH-\vt}  + \frac{\nu + \alpha q}{1+\alpha} q^{\HH} = 0.
\end{align*}
\textbf{Step 2} (General $J$): Using fusion, we decompose the general $J$ vertex with input $(v, h)$ into a column of $J =1$ vertices with input $(v, h_1, \dots, h_J)$, where $(h_1, \dots h_J)$ is weighted by $\Lambda(h; h_1, \dots, h_J)$, see Figure \ref{fig:fusion}. Define  
$H_i, H'_i, i = 0, 1, \dots, J$ in the way that  
\begin{align}
\label{eq:relation1}
&H_0 = H(x, y), \qquad H'_0 = H(x+1, y), 
\\
\label{eq:relation2}
&H_i = H_0 + \sum_{j=1}^i h_j, \qquad H'_i =  H'_0 + \sum_{j=1}^i h'_j.
\end{align}
Since $h  = h_1 + \dots + h_J$,  $H_J = H(x, y+1)$. Furthermore,  
$H_J' = H(x+1, y+1)$ in law. 
\begin{figure}[ht]
	\centering
	\includegraphics[scale = 1.5]{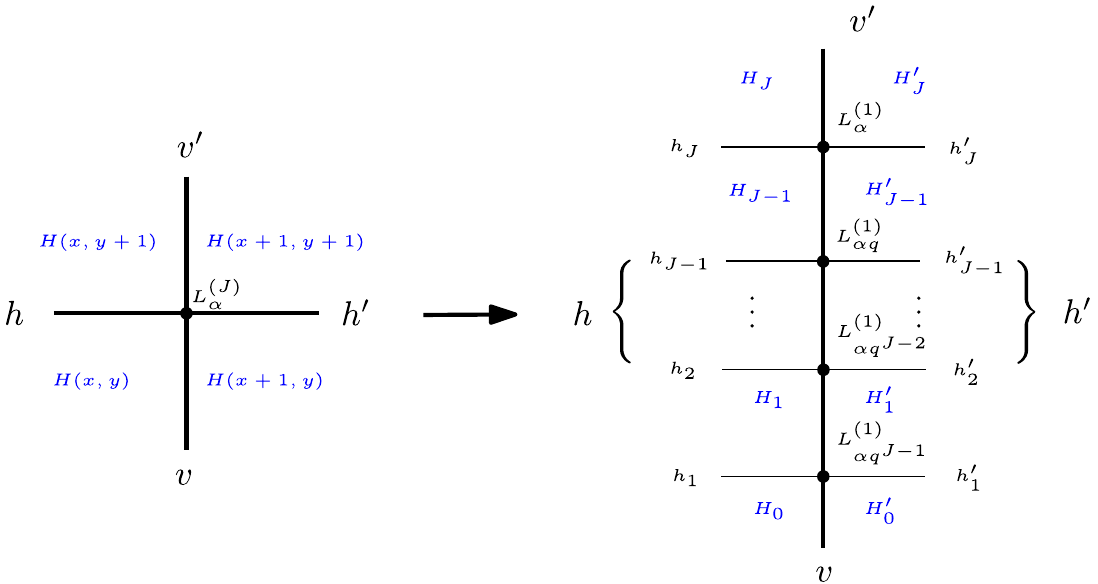}
	\caption{Given $H(x, y) = H_0, H(x+1, y) = H'_0,  H(x, y+1) = H_J$. By fusion (the spectral parameters have been reversed upside down thanks to Lemma \ref{lem:vertexid}), we have the distributional identity $H(x+1, y+1) = H_J'$. The advantage of utilizing fusion is that we can apply $J = 1$ version of \eqref{eq:fourpoint1} to each vertex in the column, where the heights around the $i$-th vertex are $H_{i-1}, H'_{i-1}, H_{i}, H'_{i}$. The horizontal input $(h_1, \dots, h_J)$ is weighted by $\widetilde{\Lambda}(h; h_1, \dots, h_J)$.}
	\label{fig:fusion}
\end{figure}
It suffices to prove 
\begin{align*}
\mathbb{E}\Big[q^{H_J'} - \frac{\alpha + \nu}{1 + \alpha} q^{H_J} - \frac{1 + \alpha q^J}{1 + \alpha} q^{H_0'} + \frac{\nu+\alpha q^J}{1+ \alpha} q^{H_0} \big| \FF\Big] = 0
\end{align*}
This is equivalent to
\begin{align}\label{eq:temp3}
\mathbb{E}\Big[ q^{H_J'} - \frac{\alpha + \nu}{1 + \alpha} q^{H_J} \big| \FF\Big] = \mathbb{E}\Big[\frac{1 + \alpha q^J}{1 + \alpha} q^{H_0'} - \frac{\nu+\alpha q^J}{1 + \alpha} q^{H_0} \big| \FF\Big].
\end{align} 
We define the sigma algebra $\mathcal{F}_{i} = \sigma\big(H_i, H_i', H_{i+1}\big)$ for $i = 0, 1, \dots, J-1$.
Since all the vertices are of horizontal spin $1/2$ now, using the $J = 1$ version of \eqref{eq:fourpoint1} (proved in \textbf{Step 1}) for the $i$-th vertex (with the spectral parameter $\alpha q^{J-i}$) looking from the bottom, we have 
\begin{align*}
&\mathbb{E} \Big[q^{H'_i} - \frac{\nu + \alpha q^{J-i}}{1 + \alpha q^{J-i}} q^{H_i} - \frac{1+\alpha q^{J+1-i}}{1 + \alpha q^{J-i}} q^{H'_{i-1}} + \frac{\nu + \alpha q^{J+1-i}}{1 + \alpha q^{J-i}} q^{H_{i-1}} \big| \FF\Big]\\
&= \mathbb{E} \bigg[ \mathbb{E}\Big[q^{H'_i} - \frac{\nu + \alpha q^{J-i}}{1 + \alpha q^{J-i}} q^{H_i} - \frac{1+\alpha q^{J+1-i}}{1 + \alpha q^{J-i}} q^{H'_{i-1}} + \frac{\nu + \alpha q^{J+1-i}}{1 + \alpha q^{J-i}} q^{H_{i-1}} \big| \mathcal{F}_{i-1}\Big]  \Big| \FF\bigg]
= 0.
\end{align*}
In other words, 
\begin{equation*}
\mathbb{E}\Big[q^{H'_i} - \frac{\nu + \alpha q^{J-i}}{1 + \alpha q^{J-i}} q^{H_i} \big| \FF\Big] = \frac{1 + \alpha q^{J+1-i}}{1 + \alpha q^{J-i}} \mathbb{E}\Big[q^{H'_{i-1}} - \frac{\nu + \alpha q^{J+1-i}}{1 + \alpha q^{J+1-i}} q^{H_{i-1}} \big| \FF\Big].
\end{equation*} 
Iterating the above equation from $i = J$ to $i =1$, one concludes the desired \eqref{eq:temp3}.
\end{proof}
\noindent To prove relation \eqref{eq:introfourpointquad1}, we need the following fact which says that under our scaling \eqref{eq:scaling}, it is unlikely that a vertex will change the direction of lines entering into it. More specifically, if a vertex has $i$ vertical input lines and $j$ horizontal input lines, with probability going to $1$, it produces $i$ vertical and $j$ horizontal output lines.
\bigskip
\\
We use $\mathcal{O}(a)$ to denote some quantity bounded by a constant times $a$, when the scaling parameter $L$ is large.
\begin{lemma}\label{lem:scalingmatrix}
	For any fixed $i_1, i_2 \in \{0, 1, \dots, I\}$ and $j_1, j_2 \in \{0,1, \dots , J\}$, as $L \to \infty$
	\begin{equation*}
	L_{\alpha}^{(J)}(i_1, j_1; i_2, j_2) = \mathbf{1}_{\{i_1 = i_2, j_1 = j_2\}} + \mathcal{O}(L^{-1}).
	\end{equation*}
\end{lemma}
\begin{proof}
Via Lemma \ref{lem:fusion}, it suffices to show that for every $i \in \{0, 1, \dots, J-1\}$ 
\begin{equation}\label{eq:claim}
L^{(1)}_{\alpha q^i}(i_1, j_1; i_2, j_2) = \mathbf{1}_{\{i_1 = i_2, j_1 = j_2\}} + \mathcal{O}(L^{-1}).
\end{equation}
Indeed, by direct computation,  under our scaling,
\begin{align*}
&L^{(1)}_{\alpha q^i} (m, 0; m, 0) = \frac{1 + \alpha q^{m+i}}{1 + \alpha q^i} = 1 - \frac{\beta_2 m}{L} + \mathcal{O}(L^{-2}), \hspace{3em}  L^{(1)}_{\alpha q^i} (m, 0; m-1, 1) = \frac{\beta_2 m}{L} + \mathcal{O}(L^{-2}),\\
&L^{(1)}_{\alpha q^i} (m, 1; m, 1) = \frac{\alpha q^i + \nu q^m}{1 + \alpha q^i} = 1 + \frac{\beta_1 (m-I)}{L} + \mathcal{O}(L^{-2}), \hspace{1em} L^{(1)}_{\alpha q^i} (m, 1; m+1, 0) = \frac{\beta_1 (I-m)}{L} + \mathcal{O}(L^{-2}),
\end{align*}
which implies \eqref{eq:claim}.
\end{proof}
\begin{theorem}\label{thm:fourpointquad}
Define 
\begin{equation*}
\Delta_x := q^{H(x+1, y)} - q^{H(x, y)},\qquad \Delta_y := q^{H(x, y+1)} - q^{H(x, y)}.
\end{equation*}
Fix $A, B > 0$, under scaling \eqref{eq:scaling}, for any $x \in [0, LA] \cap \mathbb{Z}$ and $y \in [0, LB] \cap \mathbb{Z}$ and $L > 1$,
\begin{align*}
&\mathbb{E}\Big[\xi(x+1, y+1)^2 \big| \mathcal{F}(x,y)\Big] \\
&= L^{-1} (\beta_1 + \beta_2) \Delta_x \Delta_y + J L^{-2} (\beta_2 - \beta_1) \beta_2 q^{H(x, y)} \Delta_x + I L^{-2} (\beta_1 - \beta_2) \beta_1 q^{H(x, y)} \Delta_y + \mathbf{R}(x, y),
\end{align*}
where $\mathbf{R}(x, y)$ is a random field with the uniform upper bound 
\begin{equation}\label{eq:remainderbound}
|\mathbf{R}(x, y)| \leq C L^{-4}, 
\end{equation}
for all $x \in [0, LA] \cap \mathbb{Z}$ and $y \in [0, LB] \cap \mathbb{Z}$, $C$ is some constant that only depends on $A, B$.
\end{theorem}
\begin{proof}
We only need to show that the random field $\mathbf{R}(x, y)$ defined via
\begin{equation*}
\mathbf{R}(x, y) = \mathbb{E}\Big[\xi(x+1, y+1)^2 \big| \mathcal{F}(x,y)\Big] - L^{-1} (\beta_1 + \beta_2) \Delta_x \Delta_y - J L^{-2} (\beta_2 - \beta_1) \beta_2 q^{H(x, y)} \Delta_x - I L^{-2} (\beta_1 - \beta_2) \beta_1 q^{H(x, y)} \Delta_y
\end{equation*}
satisfies \eqref{eq:remainderbound}. Using same notation as in the proof of Theorem \ref{thm:fourpoint},
\begin{equation*}
\xi:= \xi(x+1, y+1), \qquad  \qquad \mathcal{F}:=  \sigma\big(\mathsf{H}, h, v\big),
\end{equation*} 
and
\begin{equation*}
\HH := H(x, y),\qquad h := H(x, y+1) - H(x, y), \qquad  v := H(x, y) - H(x+1, y).
\end{equation*}
It is clear that 
$
\mathbb{E}\big[\xi(x+1, y+1)^2 | \mathcal{F}(x,y)\big]  = \mathbb{E}\big[\xi^2 | \mathcal{F} \big].
$
Our proof is divided into two steps.
\bigskip
\\
\textbf{Step 1} $(J=1)$: When $J=1$, $h \in \{0, 1\}$. We discuss the $h = 0$ and $h = 1$ case separately.
\bigskip
\\
If $h = 0$, 
\begin{align*}
\mathbb{E}\Big[\xi^2 \big| \FF \Big]
&= \mathbb{E}\Big[\Big(q^{H(x+1, y+1)} - \frac{\alpha + \nu}{1 + \alpha} q^{\HH} - \frac{1 + \alpha q}{1 + \alpha} q^{\HH- \vt} + \frac{\nu+\alpha q}{1+ \alpha} q^{\HH}\Big)^2 \big| \FF \Big]\
\end{align*}
Referring to \eqref{eq:distributeone}, we have (recall $\nu = q^{-I}$)
\begin{align*}
\mathbb{E}\Big[\xi^2 \big| \FF\Big]  &= \frac{1 + \alpha q^\vt}{1 + \alpha} \Big(q^{\HH-\vt} - \frac{\alpha + \nu}{1 + \alpha} q^{\HH} - \frac{1 + \alpha q}{1 + \alpha} q^{\HH-\vt} + \frac{\nu + \alpha q}{1 + \alpha} q^\HH \Big)^2\\ 
&\quad  + \frac{\alpha (1-q^\vt)}{1 + \alpha} \Big(q^{\HH-\vt+1} - \frac{\alpha + \nu}{1 + \alpha} q^{\HH} -  \frac{1 + \alpha q}{1 + \alpha} q^{\HH - v} + \frac{\nu + \alpha q}{1 + \alpha} q^{\HH} \Big)^2\\
\numberthis \label{eq:temp4} 
&=\frac{\alpha (q-1)^2 q^{-2 \vt} (1 - q^\vt) (1 + \alpha q^\vt)}{(1 + \alpha)^2} q^{2\mathsf{H}},
\end{align*}
The second equality in the above display follows from a straightforward calculation. 
\bigskip
\\
Let $\b = \frac{1 + \alpha q}{1 + \alpha}$ and rewrite \eqref{eq:temp4} as 
\begin{equation}\label{eq:temp15}
\mathbb{E}\Big[\xi^2 \big| \mathcal{F}\Big] = (1-\b) q^{-2\vt} (q^\vt -1) \big(-\b + q + (\b-1) q^\vt \big) q^{2\mathsf{H}}
\end{equation}
Referring to scaling \eqref{eq:scaling}, we see that $q^{\HH} = e^{\frac{\beta_1 - \beta_2}{L} \HH}$ is bounded, since for $x \in [0, LA]$ and $y \in [0, LB]$ $|\HH| = |H(x, y)| 
\leq L(A+B). $
In addition,
\begin{equation}\label{eq:temp1}
q =  1+ \frac{\beta_1- \beta_2}{L} + \mathcal{O}(L^{-2}), \qquad \b = 1 - \frac{\beta_2}{L}   + \mathcal{O}(L^{-2})
\end{equation}
Using the expansion of $q$ and $\b$ in \eqref{eq:temp1}, we have
\begin{align}\label{eq:temp5}
\mathbb{E}\Big[\xi^2 \big| \FF\Big] = q^{2\HH}\beta_2 (\beta_2 - \beta_1)^2 \vt L^{-3} + \mathcal{O}(L^{-4}). 
\end{align}
When $h = 0$, $\Delta_y = q^{H(x, y+1)} - q^{H(x, y)} = 0$. Under scaling \eqref{eq:scaling}, 
\begin{equation*}
\Delta_x = q^{H(x+1, y)} - q^{H(x, y)} = q^{\HH} (q^{-\vt} - 1) = q^{\HH} \frac{\vt (\beta_2 - \beta_1)}{L} + \mathcal{O}(L^{-2}). 
\end{equation*}
Thereby,
\begin{align*}
&L^{-1} (\beta_1 + \beta_2) \Delta_x \Delta_y + L^{-2} (\beta_2 - \beta_1) \beta_2 q^{\HH} \Delta_x + I L^{-2} (\beta_1 - \beta_2) \beta_1 q^{\HH} \Delta_y\\
&= L^{-2} (\beta_2 - \beta_1) \beta_2 q^{\HH} \Delta_x,\\
\numberthis \label{eq:temp6}
&= q^{2\HH}  \beta_2 (\beta_2 - \beta_1)^2 \vt L^{-3} + \mathcal{O}(L^{-4}).
\end{align*}
It follows from \eqref{eq:temp5} and \eqref{eq:temp6} (note that $J  = 1$)
\begin{align*}
\mathbf{R}(x, y) = \mathbb{E}\Big[\xi^2 \big| \FF\Big] -  \Big(L^{-1} (\beta_1 + \beta_2) \Delta_x \Delta_y + L^{-2} (\beta_2 - \beta_1) \beta_2 q^{\HH} \Delta_x + I L^{-2} (\beta_1 - \beta_2) \beta_1 q^{\HH} \Delta_y\Big) = \mathcal{O}(L^{-4}). 
\end{align*} 
If $h=1$, $H(x+1, y+1)$ is distributed as \eqref{eq:distributetwo}, then  (recall $\nu = q^{-I}$)
\begin{align*}
\mathbb{E}\Big[\xi^2 \big| \FF\Big] 
& = \frac{1 - \nu q^\vt}{1 + \alpha} \Big(q^{\HH-\vt} - \frac{\alpha + \nu}{1 + \alpha} q^{\HH+1} - \frac{1+ \alpha q}{1 + \alpha} q^{\HH-\vt} + \frac{\nu + \alpha q }{1 + \alpha} q^\HH \Big)^2, \\ 
&\quad + \frac{\alpha + \nu q^\vt}{1 + \alpha} \Big(q^{\HH+1-\vt} - \frac{\alpha + \nu}{1 + \alpha} q^{\HH+1} - \frac{1 +\alpha q}{1 + \alpha } q^{\HH-\vt}
+ \frac{\nu + \alpha q}{1 + \alpha} q^\HH \Big)^2,\\
&= \frac{(q-1)^2 q^{-2(I+\vt)}(q^I - q^\vt)(\alpha q^I + q^\vt)}{(1+\alpha)^2} q^{2\HH}
\end{align*}
Rewrite the RHS above as (recall $\b = \frac{1 + \alpha q}{1 + \alpha}$)
\begin{equation*}
\mathbb{E}\Big[\xi^2 \big| \mathcal{F}\Big] = (q-\b) q^{-2 (I + \vt)} (q^I - q^\vt) \big((-1 + \b) q^I + q^\vt (q-\b)\big) q^{2\HH}
\end{equation*}
Using the expansion in \eqref{eq:temp1}, we deduce
\begin{equation}\numberthis \label{eq:quadexpan1}
\mathbb{E} \Big[\xi^2 \big| \FF\Big] =  q^{2\HH} (I-\vt) (\beta_2 - \beta_1)^2 \beta_1 L^{-3}  + \mathcal{O}(L^{-4}).
\end{equation}
When $h = 1$,
\begin{equation*}
\Delta_x =  q^{\HH} (\beta_2 - \beta_1) \vt L^{-1} + \mathcal{O}(L^{-2})\qquad \Delta_y = q^{\HH} (\beta_1 - \beta_2) L^{-1} + \mathcal{O}(L^{-2}),
\end{equation*} 
which yields
\begin{align*}
&L^{-1} (\beta_1 + \beta_2) \Delta_x \Delta_y + L^{-2} (\beta_2 - \beta_1) \beta_2 q^{\HH} \Delta_x + I L^{-2} (\beta_1 - \beta_2) \beta_1 q^{\HH} \Delta_y\\
\numberthis \label{eq:temp7}
&= q^{2\HH} (I-\vt) (\beta_2 - \beta_1)^2 \beta_1 L^{-3} + \mathcal{O}(L^{-4}).
\end{align*}
Combining \eqref{eq:quadexpan1} and \eqref{eq:temp7} yields 
\begin{align*}
\mathbf{R}(x, y) &= \mathbb{E}\Big[\xi^2 \big| \FF\Big] - \Big( L^{-1} (\beta_1 + \beta_2) \Delta_x \Delta_y + L^{-2} (\beta_2 - \beta_1) \beta_2 q^{\HH} \Delta_x + I L^{-2} (\beta_1 - \beta_2) \beta_1 q^{\HH} \Delta_y\Big) =  \mathcal{O}(L^{-4}).
\end{align*}
This concludes \eqref{eq:remainderbound}.
\bigskip\\
\textbf{Step 2} (general $J$): Similar as what we did in Theorem \ref{thm:fourpoint}, we apply fusion (see Figure \ref{fig:fusion}). Recall $H_i$, $H_i'$ from \eqref{eq:relation1} and \eqref{eq:relation2} and define 
\begin{align*}
\xi_i &:= q^{H'_i} - \frac{\nu + \alpha q^{J-i}}{1 + \alpha q^{J-i}} q^{H_i} - \frac{1 + \alpha q^{J-i + 1}}{1 + \alpha q^{J-i}} q^{H'_{i-1}} +  \frac{\nu + \alpha q^{J+1-i}}{1 + \alpha q^{J-i}} q^{H_{i-1}},\\
&= q^{H'_i} - \frac{\nu + \alpha q^{J-i}}{1 + \alpha q^{J-i}} q^{H_i} - \frac{1 + \alpha q^{J-i+1}}{1 + \alpha q^{J-i}}\Big(q^{H'_{i-1}} - \frac{\nu + \alpha q^{J+1-i}}{1 + \alpha q^{J+1-i}} q^{H_{i-1}}\Big).
\end{align*} 
By straightforward calculation,
\begin{align*}
\sum_{i=1}^J \frac{1 + \alpha q^{J-i}}{1 + \alpha} \xi_i &= \sum_{i=1}^J \bigg( \frac{1+\alpha q^{J-i}}{1 + \alpha} \Big(q^{H'_i} - \frac{\nu + \alpha q^{J-i}}{1 + \alpha q^{J-i}} q^{H_i}\Big) - \frac{1+\alpha q^{J-i+1}}{1 + \alpha} \Big(q^{H'_{i-1}} - \frac{\nu + \alpha q^{J+1-i}}{1 + \alpha q^{J+1-i}} q^{H_{i-1}}\Big) \bigg) \\
&= q^{H_J'} - \frac{\nu + \alpha}{1 + \alpha} q^{H_J} -  \frac{1 + \alpha q^J}{1 + \alpha} \Big(q^{H'_0} - \frac{\nu + \alpha q^{J}}{1 + \alpha q^{J}} q^{H_0}\Big)
\numberthis \label{eq:temp8}
= \xi,
\end{align*}
where the second equality follows from the telescoping property of the summation.
\bigskip
\\
By Theorem \ref{thm:fourpoint},  $\xi_i$ are martingale increments, so $\mathbb{E}\big[\xi_i \xi_j | \mathcal{F}\big] = 0$ for $i \neq j$. It follows from \eqref{eq:temp8} that 
\begin{align}
\label{eq:temp9}
\mathbb{E}\Big[\xi^2 \big| \FF\Big] 
= \mathbb{E}\Big[\sum_{i=1}^J \Big(\frac{1 + \alpha q^{J-i}}{1 + \alpha} \xi_i\Big)^2 \big| \FF \Big] = \sum_{i=1}^{J} \Big(\frac{1 + \alpha q^{J-i}}{1 + \alpha}\Big)^2 \mathbb{E}\Big[\xi_i^2 \big| \FF\Big].
\end{align}
Using  the $J = 1$ version of \eqref{eq:temp15} proved in \textbf{Step 1}  for the $i$-th vertex counting from bottom (here, though the spectral parameter changes from $\alpha$ to $\alpha q^i$, it does not matter under our scaling)
\begin{align}\label{eq:temp38}
\mathbb{E}\Big[\xi_i^2 \big| \mathcal{F}_{i-1}\Big] = L^{-1} (\beta_1 + \beta_2)\Delta^i_x \Delta^i_y +  L^{-2} (\beta_2 - \beta_1) \beta_2 q^{\HH} \Delta_x^i +  I L^{-2} (\beta_1 - \beta_2) \beta_1 q^{\HH} \Delta^i_y + \mathbf{R}_i (x, y)
\end{align}
where $\Delta_x^{i} = q^{H'_{i-1}} - q^{H_{i-1}}$ and $\Delta_y^i = q^{H_i} - q^{H_{i-1}}$ and, also recall that  
$\mathcal{F}_i = \sigma\big(H_i, H_{i+1}, H'_i\big)
$. By \textbf{Step 1}, there exists constant $C$ only depending on $A, B$ such that
\begin{equation}\label{eq:temp}
\sup_{\substack{i \in \{1, \dots ,J\}\\ (x, y) \in [0,LA] \times [0, LB]}}|\mathbf{R}_i (x, y)| \leq C L^{-4}.
\end{equation}
By conditioning,
$\mathbb{E}\big[\xi^2_i \big| \FF\big] = \mathbb{E}\Big[ \mathbb{E}\big[\xi_i^2 | \mathcal{F}_{i-1}\big] \big| \mathcal{F}\Big]$ (note that here we are not using the tower property but instead the sequential update rule). Using \eqref{eq:temp9} and \eqref{eq:temp38}, we get
\begin{align*}
\mathbb{E}\Big[\xi^2 \big| \FF\Big] 
&= \sum_{i=1}^J \Big(\frac{1 + \alpha q^{J-i}}{1 +\alpha}\Big)^2 \mathbb{E}\Big[L^{-1} (\beta_1 + \beta_2)\Delta^i_x \Delta^i_y +  L^{-2} (\beta_2 - \beta_1) \beta_2 q^{\HH} \Delta_x^i\\ 
&\hspace{10em} + I L^{-2} (\beta_1 - \beta_2) \beta_1 q^{\HH} \Delta^i_y + \mathbf{R}_i (x, y) \big| \FF\Big]. 
\end{align*}
Note that under our scaling, $\lim_{L \to \infty} \frac{1 + \alpha q^{J-i}}{1 + \alpha} = 1$, along with \eqref{eq:temp},
\begin{align*}
\mathbb{E}\Big[\xi^2 \big| \FF\Big] 
\numberthis \label{eq:temp2} = \sum_{i=1}^J \mathbb{E}\Big[L^{-1} (\beta_1 + \beta_2)\Delta^i_x \Delta^i_y +  L^{-2} (\beta_2 - \beta_1) \beta_2 q^{\HH} \Delta_x^i +  I L^{-2} (\beta_1 - \beta_2) \beta_1 q^{\HH} \Delta^i_y \big| \FF\Big] + \mathcal{O}(L^{-4}).
\end{align*}
It is clear that
\[\sum_{i=1}^J \Delta^i_y = \sum_{i=1}^J \Big(q^{H_i} - q^{H_{i-1}}\Big) = q^{H_J} - q^{H_0} = \Delta_y.\]
Furthermore, by Lemma \ref{lem:scalingmatrix},
\begin{equation*}
\mathbb{P}\Big(\exists\, i \text{ such that } \Delta_x^i \neq \Delta_x \big| \FF\Big) = 1 - \mathcal{O}(L^{-1}).
\end{equation*}
Hence, we can simplify \eqref{eq:temp2} and get
\begin{align*}
\mathbb{E}\Big[\xi^2 \big| \FF\Big]
&= \sum_{i=1}^J \mathbb{E}\Big[L^{-1} (\beta_1 + \beta_2)\Delta_x \Delta^i_y +  L^{-2} (\beta_2 - \beta_1) \beta_2 q^{\HH} \Delta_x + I L^{-2} (\beta_1 - \beta_2) \beta_1 q^{\HH} \Delta^i_y \big| \FF\Big] + \mathcal{O}(L^{-4}),
\\
&= \mathbb{E}\Big[L^{-1} (\beta_1 + \beta_2) \Delta_x \Delta_y + J L^{-2} (\beta_2 - \beta_1) \beta_2 q^{\HH} \Delta_x + I L^{-2} (\beta_1 - \beta_2) \beta_1 q^{\HH} \Delta_y \big| \FF \Big] + \mathcal{O}(L^{-4}),\\
&= L^{-1}(\beta_1 + \beta_2) \Delta_x \Delta_y + J L^{-2}(\beta_2 - \beta_1) \beta_2 q^{\HH} \Delta_x + I L^{-2} (\beta_1 - \beta_2) \beta_1 q^{\HH} \Delta_y + \mathcal{O}(L^{-4}).
\end{align*}
The last line is because $\Delta_x$ and $\Delta_y$ and $\HH$ are measurable with respect to $\mathcal{F}$.
\end{proof}
\begin{remark}\label{rmk:quadratic}
The identity \eqref{eq:introfourpointquad} which holds for stochastic six vertex model no long works for the SHS6V model. For example, consider $I = 2$ and $J =1 $. For an arbitrary vertex $(x, y)$, if there exists three parameters $\gamma_1, \gamma_2, \gamma_3$ such that \eqref{eq:introfourpointquad} is true. When $h = 0$, referring to \eqref{eq:temp4}, we have
\begin{align*}
\mathbb{E}\Big[\xi(x+1, y+1)^2 \big| \mathcal{F}(x,y)\Big]
= \frac{\alpha (q-1)^2 q^{-2 \vt} (1 - q^\vt) (1 + \alpha q^\vt)}{(1 + \alpha)^2} q^{2\mathsf{H}}
\end{align*}
	Since $\Delta_y = 0$, the right hand side of \eqref{eq:introfourpointquad} reduces to
	\begin{equation*}
	\gamma_1 \Delta_x \Delta_y + \gamma_2 \Delta_x q^{\HH} + \gamma_3 \Delta_y q^{\HH} = \gamma_2 \Delta_x q^{\HH} = \gamma_2 (q^{-\vt} - 1) q^{2\HH} 
	\end{equation*}
	So for all $v \in \{0, 1, \dots, I\}$,
	\begin{equation*}
\frac{\alpha (q-1)^2 q^{- \vt} (q^{-\vt} -1) (1 + \alpha q^\vt)}{(1 + \alpha)^2} q^{2\mathsf{H}} = \gamma_2 (q^{-\vt} - 1) q^{2\HH}
	\end{equation*}
	Canceling the factor $(q^{-\vt} - 1) q^{2\HH}$ on both sides, we get
	\begin{equation*}
	\frac{\alpha (q-1)^2 q^{- \vt} (1 + \alpha q^\vt)}{(1 + \alpha)^2} = \gamma_2
	\end{equation*}
	Since $\gamma_2$ does not depend on $\vt$, so the previous equation could not hold for $\vt = 1, 2$ simultaneously.
\end{remark}
\noindent The following corollary is a direct consequence of Theorem \ref{thm:fourpointquad}.
\begin{cor}\label{cor:quadbound}
Fix $A, B > 0$, there exists constant $C$ s.t. for every $x \in [0, LA] \cap \mathbb{Z}, y \in [0, LB] \cap \mathbb{Z}$ and $L > 1$
\[\mathbb{E} \Big[\xi(x+1, y+1)^2 \big| \mathcal{F}(x,y)\Big] \leq C L^{-3}.\]
\end{cor}
\begin{proof}
It is clear that there exists $C$ such that for any $x \in [0, LA] \cap \mathbb{Z}$ and $y \in [0, LB] \cap \mathbb{Z}$,
\begin{align*}
|\Delta_x| &= \Big|q^{H(x+1, y)} - q^{H(x, y)}\Big|
= q^{H(x, y)} \Big|e^{\frac{(\beta_1 - \beta_2) h }{L}} -  1\Big| \leq C L^{-1}.
\end{align*}
Similarly,  $|\Delta_y| \leq C L^{-1}$. Referring to Theorem \ref{thm:fourpointquad} (note that $q^{H(x, y)}$ is bounded), the corollary follows.
\end{proof}
\section{Proof of the main results}
\label{sec:proofofresults}
\noindent Having established the four point relation, we move on proving Theorem \ref{thm:lln} and Theorem \ref{thm:clt}. Corollary \ref{cor:clt} follows from a straightforward argument once we proved Theorem \ref{thm:clt}. For the ensuing discussion, we
will usually write $C$ for constants, we might not generally specify when irrelevant terms are being absorbed into
the constants. We might also write for example $C(n)$ when we want to specify which parameter the constant depends on.

\begin{proof}[Proof of Theorem \ref{thm:lln}]
Given Theorem \ref{thm:fourpoint}, our proof is akin to \cite[Theorem 5.1]{BG19}. We provide the detail for the sake of completeness. Recall $\mathfrak{q} = q^{\frac{1}{L}}$, to prove $\frac{1}{L} H(Lx, Ly) \to \mathbf{h}(x, y)$ uniformly in probability for $x \in [0, A]$ and $y \in [0, B]$, it suffices to show that 
$q^{H(Lx, Ly)} \to \mathfrak{q}^{\mathbf{h}(x, y)}$
uniformly in probability. To this end, we write $$q^{H(Lx, Ly)} = \mathbb{E}\Big[q^{H(Lx, Ly)}\Big] + q^{H(Lx, Ly)} - \mathbb{E}\Big[q^{H(Lx, Ly)}\Big].$$
It suffices to show that as $L \to \infty$,
\begin{enumerate}[(i).]
\label{itm:lln1}
\item $\mathbb{E}[q^{H(Lx, Ly)}] \to \mathfrak{q}^{\mathbf{h}(x, y)}$  uniformly for $(x, y) \in [0, A] \times [0, B]$, 
\label{itm:lln2}
\item $q^{H(Lx, Ly)} - \mathbb{E}[q^{H(Lx, Ly)}] \to 0$ uniformly in probability for $(x, y) \in [0, A] \times [0, B]$.  
\end{enumerate}
We first demonstrate (i). By Theorem \ref{thm:fourpoint}, 
\begin{equation*}
\mathbb{E}\Big[q^{H(x+1, y+1)}\Big] - b_1 \mathbb{E}\Big[q^{H(x, y+1)}\Big] - b_2 \mathbb{E}\Big[q^{H(x+1, y)}\Big] + (b_1 + b_2 - 1) \mathbb{E}\Big[q^{H(x, y)}\Big] = 0,
\end{equation*}
where $b_1 = \frac{\alpha + \nu}{1 + \alpha}$, $b_2 = \frac{1 + \alpha q^J}{1 + \alpha}$. Summing this equation over $x = 0, 1, \dots, LX-1$ and $y = 0, 1, \dots LY-1$ yields
\begin{align*}
&-(1-b_1) \sum_{x =1}^{LX-1} \mathbb{E}\Big[q^{H(x, 0)}\Big] - (1-b_2) \sum_{y = 1}^{LY-1} \mathbb{E}\Big[q^{H(0, y)}\Big] 
+ (1-b_1)  \sum_{x = 1}^{LX-1} \mathbb{E}\Big[q^{H(x, LY)}\Big]\\ 
\numberthis \label{eq:temp17}
&+ (1-b_2) \sum_{y=1}^{LY-1} \mathbb{E}\Big[q^{H(LX, y)}\Big] 
+ (b_1+ b_2 -1) \mathbb{E}\Big[q^{H(0, 0)}\Big] - b_2 \mathbb{E}\Big[q^{H(LX, 0)}\Big] - b_1 \mathbb{E}\Big[q^{H(0, LY)}\Big] + \mathbb{E}\Big[q^{H(LX, LY)}\Big] = 0
\end{align*}
Since $H$ is Lipschitz, the sequence of deterministic functions $\mathbb{E}[q^{H(L\cdot, L\cdot)}] = \mathbb{E}[\mathfrak{q}^{\frac{1}{L} H(L\cdot, L\cdot)}]  \in C([0, A] \times [0, B])$  is uniformly bounded and equi-continuous. By Arzela-Ascoli Theorem, it has a limit point $\mathfrak{q}^{\widetilde{h}}$.
\bigskip
\\
Under scaling \eqref{eq:scaling}, when $L \to \infty$,
\begin{equation}\label{eq:scalingb1b2}
b_1 = 1 - \beta_1 I L^{-1} + \mathcal{O}(L^{-2}), \qquad b_2 = 1 - \beta_2 J L^{-1} + \mathcal{O}(L^{-2}).
\end{equation} 
Combining this with \eqref{eq:temp17} and taking the $L \to \infty$ limit, $\mathfrak{q}^{\widetilde{\mathbf{h}}}$ satisfies the integral equation
\begin{align*}
&-I \beta_1 \int_0^X \mathfrak{q}^{\widetilde{\mathbf{h}}(x, 0)} dx - J \beta_2  \int_0^Y \mathfrak{q}^{\widetilde{\mathbf{h}}(0, y)} + I\beta_1  \int_0^X \mathfrak{q}^{H(x, Y)} dx + J\beta_2 \int_0^Y \mathfrak{q}^{H(X, y)} dy\\ 
&+ \mathfrak{q}^{\widetilde{\mathbf{h}}(0, 0)} - \mathfrak{q}^{\widetilde{\mathbf{h}}(X, 0)} - \mathfrak{q}^{\widetilde{\mathbf{h}}(0, Y)} + \mathfrak{q}^{\widetilde{\mathbf{h}}(X, Y)} = 0
\end{align*}
In other words, any limit point $\mathfrak{q}^{\widetilde{\mathbf{h}}}$ of $\mathbb{E}\big[\mathfrak{q}^{\frac{1}{L} H(Lx, Ly)}\big]$ as $L \to \infty$ satisfies the telegraph equation
\begin{equation*}
\frac{\partial^2}{\partial x \partial y} \mathfrak{q}^{\widetilde{\mathbf{h}}(x, y)} + I\beta_1 \frac{\partial}{\partial y} \mathfrak{q}^{\widetilde{\mathbf{h}}(x, y)} + J\beta_2 \frac{\partial}{\partial x} \mathfrak{q}^{\widetilde{\mathbf{h}}(x, y)} = 0.
\end{equation*}
By our assumption on the boundary, we also know that $\mathfrak{q}^{\widetilde{\mathbf{h}}(x, 0)} = \mathfrak{q}^{\chi(x)}$ and $\mathfrak{q}^{\widetilde{\mathbf{h}}(0, y)} = \mathfrak{q}^{\psi(y)}$. This implies that $\widetilde{\mathbf{h}} = \mathbf{h} $, which concludes (i).
\bigskip
\\  
To verify (ii), we denote by $U(x, y) = q^{H(x, y)} - \mathbb{E}\big[q^{H(x, y)}\big]$. Using Theorem \ref{thm:fourpoint},  
$q^{H(x, y)}$ and $\mathbb{E}\big[q^{H(x, y)}\big]$ satisfy the discrete telegraph equation \eqref{eq:discretetelegraph} with $g$ given by $\xi$ and $0$ respectively, hence by linearity,
\begin{equation*}
U(x+1, y+1) - b_1 U(x, y+1) - b_2 U(x+1, y) + (b_1 + b_2 - 1) U(x, y) = \xi(x+1, y+1).
\end{equation*}
Summing over $x = 0, 1, \dots, LX-1$ and $y = 0, 1, \dots LY-1$, along with the fact  $U(x, 0) = U(0, y) = 0$ yields
\begin{align}\label{eq:temp31}
U(LX, LY) + (1-b_1) \sum_{x=1}^{LX-1} U(x, LY) + (1-b_2)\sum_{y=1}^{LY-1} U(LX, y) = \sum_{x=1}^{LX} \sum_{y=1}^{LY} \xi(x, y). 
\end{align}
Since $\xi(x, y)$ is a martingale increment,  using  Corollary \ref{cor:quadbound}
\begin{equation*}
\mathbb{E}\bigg[\Big(\sum_{x= 1}^{LA} \sum_{y=1}^{LB} \xi(x, y)\Big)^2\bigg] = \sum_{x=1}^{LA} \sum_{y=1}^{LB} \mathbb{E}\Big[\xi(x, y)^2\Big] \leq C AB L^{-1}.
\end{equation*}
Applying Doob's $L^p$ maximal inequality, it is clear that 
\begin{equation}\label{eq:temp32}
\sup_{(X, Y) \in [0, A] \times [0, B]} \Big|\sum_{x = 1}^{LX} \sum_{y=1}^{LY} \xi(x, y)\Big| \overset{p}{\to} 0.
\end{equation}
Observing that $U(L\cdot, L\cdot)$ are uniformly bounded and uniformly Lipschitz on $[0, A]  \times [0, B]$. Therefore, their law are tight, any subsequential limit $\widetilde{U}$ has continuous trajectories must solve the $L = \infty$ version of \eqref{eq:temp31}, which reads (the right hand side is zero by \eqref{eq:temp32})
\begin{equation*}
\widetilde{U}(X, Y) + I\beta_1 \int_0^X \widetilde{U}(x, Y) dx + J\beta_2 \int_0^Y \widetilde{U}(X, y) dy = 0.
\end{equation*}
According to \cite[Prop 4.1]{BG19}, the only solution to the above equation  is given by $\widetilde{U} = 0$, which implies (ii).
\end{proof}  
\noindent We move on proving the functional CLT for the SHS6V model. The proof of Theorem
1.7 is composed of showing the finite dimensional weak convergence and demonstrating the tightness, which is
formulated into the following two propositions. 
\bigskip
\\
Denote by $$U_L(x, y) := \sqrt{L} \bigg( q^{H(Lx, Ly)} - \mathbb{E}\Big[q^{H(Lx, Ly)}\Big]\bigg) = \sqrt{L} U(Lx, Ly).$$
\begin{prop}[finite dimensional convergence]\label{prop:cltfinitedim}
With the same setup as in Theorem 1.7, we have the weak convergence in finite dimension as $L \to \infty$,
\begin{equation*}
U_L(x, y) \Rightarrow \varphi(x, y).
\end{equation*} 
\end{prop}
\noindent Recall that we linearly interpolate $H(x, y)$ for non-integer $x, y$, thus $H$ is a function in $C(\mathbb{R}^2_{\geq 0})$, so is $U_L (x, y)$.
\begin{prop}[tightness]\label{prop:tightness}
For each fixed $A, B > 0$ and $n \in \mathbb{N}$, there is a constant $C$ (only depends on $n, A, B$) such that for all $L > 1$ and $(X_1, Y_1), (X_2, Y_2) \in [0, LA] \times [0, LB]$,
\begin{align}\label{eq:incrementbound}
\mathbb{E}\bigg[\Big(U_L(X_1, Y_1) - U_L(X_2, Y_2)\Big)^{2n}\bigg] \leq C \big(|X_1 - X_2| + |Y_1 - Y_2|\big)^n.
\end{align}
Consequently, the sequence of random function $U_L (\cdot, \cdot) \in C(\mathbb{R}_{\geq 0}^2)$ is tight.
\end{prop}
\begin{proof}[Proof of Theorem \ref{thm:clt}]
The proof is a direct combination of Proposition \ref{prop:cltfinitedim} and Proposition \ref{prop:tightness}.
\end{proof}
We first prove the finite dimensional weak convergence.
\begin{proof}[Proof of Proposition \ref{prop:cltfinitedim}]
Recall that in the proof of Theorem \ref{thm:lln}, we set $U(x, y) = q^{H(x, y)} - \mathbb{E}[q^{H(x, y)}]$. As shown earlier, we have
\begin{equation*}
U(x+1, y+1) - b_1 U(x, y+1) - b_2 U(x+1, y) + (b_1 + b_2 - 1) U(x, y) = \xi(x+1, y+1)
\end{equation*}
Furthermore, since $H(x, 0)$ and $H(0, y)$ are deterministic, we have $U(x, 0) = U(y, 0) = 0$. By \eqref{eq:temp29}, one has
\begin{equation}\label{eq:temp10}
U(X, Y) = \sum_{x = 1}^{X} \sum_{y = 1}^{Y} \mathcal{R}^d (X, Y; x, y) \xi(x, y).
\end{equation}
Here $\mathcal{R}^d$ is defined through \eqref{eq:discreteRiemann} with $b_1 = \frac{\alpha + \nu}{1 + \alpha}$, $b_2 = \frac{1 + \alpha q^J}{1 + \alpha}$. 
\bigskip
\\
We need to show that $U_L (\cdot, \cdot) = \sqrt{L} U(L \cdot, L\cdot)$ converges weakly to $\varphi(\cdot, \cdot)$ (given by \eqref{eq:temp13}) in finite dimension. As in the proof of \cite[Theorem 6.1]{BG19}, we use the martingale central limit theorem \cite[Section 3]{HH14} for the martingale (note that $U_L(X, Y) = M_L (L^2 XY)$)
\begin{equation}\label{eq:temp24}
\bigg(M_L(t) := \sum_{i=1}^t \sqrt{L}\mathcal{R}^d(LX, LY, x(i), y(i)) \xi(x(i), y(i)), 1\leq t \leq L^2 XY\bigg)
\end{equation}
where we linearly order points in $[1, LX] \times [1, LY]$ by sequentially tracing the diagonals $x + y = \text{ const}$
\begin{equation}\label{eq:ordering}
(x(1), y(1)) := (1, 1),\quad (x(2), y(2)) := (2, 1),\quad (x(3), y(3)) := (1, 2),\quad (x(4), y(4)) := (3, 1) \dots
\end{equation}
Note that we will only deal with the one point convergence $M_L(L^2 XY) \Rightarrow \varphi(X, Y)$ for simplicity, the finite dimensional convergence can be proved by invoking multi-dimensional version of martingale CLT (see \cite[Theorem 3.1]{ST19}) for a multi-dimensional version of the martingale in \eqref{eq:temp24}.  
\bigskip
\\
The key for the proof is to study the conditional variance of $M_L(t)$ at $t = L^2 XY$. We show that as $L \to \infty$, it converges to the variance of $\varphi$ \eqref{eq:temp13} in probability. In other words, we need to prove 
\begin{align*}
&L \sum_{x = 0}^{LX-1} \sum_{y = 0}^{LY-1} \mathcal{R}^d (LX, LY, x+1, y+1)^2 \mathbb{E}\Big[\xi(x+1, y+1)^2 \big| \mathcal{F}(x,y)\Big]\\ 
\numberthis \label{eq:temp12}
&\p \int_0^X \int_0^Y \mathcal{R}_{IJ}(X, Y, x, y)^2 \Big((\beta_1 + \beta_2) \mathfrak{q}^{\mathbf{h}}_x \mathfrak{q}^{\mathbf{h}}_y + J(\beta_2 - \beta_1) \beta_2 \mathfrak{q}^{\mathbf{h}} \mathfrak{q}^{\mathbf{h}}_x + I(\beta_1 - \beta_2) \beta_1 \mathfrak{q}^{\mathbf{h}} \mathfrak{q}^{\mathbf{h}}_y\Big) dx dy.
\end{align*}
where the RHS above is the variance of $\varphi(X, Y)$, see Remark \ref{rmk:covariance}. 
\bigskip
\\
To prove this convergence, we first use Theorem \ref{thm:fourpointquad}, 
\begin{align*}
&L \sum_{x = 0}^{LX-1} \sum_{y = 0}^{LY-1} \mathcal{R}^d (LX, LY, x+1, y+1)^2 \mathbb{E}\Big[\xi(x+1, y+1)^2 \big| \mathcal{F}(x,y)\Big]\\ 
&= \!\sum_{x = 0}^{LX-1} \sum_{y = 0}^{LY-1}\!\mathcal{R}^d (LX, LY, x+1, y+1)^2 \Big((\beta_1 + \beta_2) \Delta_x \Delta_y + J L^{-1} (\beta_2 - \beta_1) \beta_2 q^{H(x, y)} \Delta_x + I L^{-1} (\beta_1 - \beta_2) \beta_1 q^{H(x, y)} \Delta_y\Big)\\ 
&\quad + L \sum_{x = 0}^{LX-1} \sum_{y = 0}^{LY-1} \mathcal{R}^d (LX, LY, x+1, y+1)^2 \mathbf{R}(x, y).
\end{align*}
By \eqref{eq:remainderbound}, $\sup_{x \in [0, LA], y \in [0, LB]} |\mathbf{R}(x, y)| \leq C L^{-4}$, together with the fact $\mathcal{R}^d$ is uniformly bounded in $[0, LA] \times [0, LB]$, we have almost surely,
\begin{equation*}
L \sum_{x = 0}^{LX - 1} \sum_{y = 0}^{LY-1} \mathcal{R}^d (LX, LY, x+1, y+1)^2 \mathbf{R}(x, y) \to 0 
\end{equation*}
uniformly in $(x, y) \in [0, LA] \times [0, LB]$. As a result, to demonstrate \eqref{eq:temp12}, it suffices to prove that as $L \to \infty$
\begin{align}
\label{eq:gradientx}
L^{-1} \sum_{x = 0}^{LX - 1} \sum_{y = 0}^{LY -1} \mathcal{R}^d (LX, LY, x+1, y+1)^2 q^{H(x, y)} \Delta_x &\p \int_0^X \int_0^Y \mathcal{R}_{IJ} (X, Y, x, y)^2 \mathfrak{q}_x^{\mathbf{h}} \mathfrak{q}^{\mathbf{h}} dx dy
\\
\label{eq:gradienty}
L^{-1}\sum_{x = 0}^{LX - 1} \sum_{y = 0}^{LY -1} \mathcal{R}^d (LX, LY, x+1, y+1)^2 q^{H(x, y)} \Delta_y &\p \int_0^X \int_0^Y \mathcal{R}_{IJ} (X, Y, x, y)^2 \mathfrak{q}_y^{\mathbf{h}} \mathfrak{q}^{\mathbf{h}} dx dy
\\
\label{eq:gradientxy}
\sum_{x = 0}^{LX-1} \sum_{y = 0}^{LY-1} \mathcal{R}^d (LX, LY, x+1, y+1)^2 \Delta_x \Delta_y  &\p \int_0^X \int_0^Y \mathcal{R}_{IJ} (X, Y, x, y)^2  \mathfrak{q}_x^{\mathbf{h}} \mathfrak{q}_y^{\mathbf{h}} dxdy
\end{align}
To demonstrate these approximations, as in the proof of \cite[Theorem 6.1]{BG19}, we split the the interval $[0, LX] \times [0, LY]$ into squares such as $[LX_0, L(X_0 + \theta)] \times [LY_0, L(Y_0 + \theta)]$ (where $\theta$ is small) and apply the discrete to continuous approximation in each square.
\bigskip
\\
We first demonstrate \eqref{eq:gradientx}, for 
$x \in [L X_0, L(X_0 + \theta)]$ and  $y \in [LY_0, L(Y_0 + \theta)]$, it is not hard to see that  $\mathcal{R}^d(LX, LY, Lx, Ly) \to \mathcal{R}_{IJ} (X, Y, x, y)$ uniformly for $0 \leq x \leq X \leq A$ and $0 \leq y \leq Y \leq B$ (see \cite[Eq 2.9]{ST19} for  $I = J = 1$ case). Thus,
\begin{equation}\label{eq:temp20}
\mathcal{R}^d (LX, LY, x, y) = \mathcal{R}_{IJ} (X, Y, X_0, Y_0) + \mathcal{O}(\theta) + o(1), \qquad
q^{H(LX, LY)}  = \mathfrak{q}^{\mathbf{h}(X_0, Y_0)} + \mathcal{O}(\theta) + o(1). 
\end{equation}
where $o(1)$ represents the term converging to zero as $L \to \infty$. Using these expansions, we have 
\begin{align*}
&L^{-1}\sum_{\substack{x \in [LX_0, L(X_0 + \theta)]\\ y \in [LY_0, L(Y_0 + \theta)]}} \mathcal{R}^d (LX, LY, x+1, y+1)^2 q^{H(x, y)} \Delta_x
\\ 
\numberthis \label{eq:temp19}
&= L^{-1} \mathcal{R}_{IJ} (X, Y, X_0, Y_0)^2 \mathfrak{q}^{\mathbf{h}(X_0, Y_0)} \times \Big(\sum_{y \in [LY_0, L(Y_0 + \theta)]} \big(q^{H(L(X_0 + \theta), y)} - q^{H(LX_0, y)}\big)\Big) + \mathcal{O}(\theta^3) + \theta^2 o(1)
\end{align*}
Using law of large number proved in Theorem \ref{thm:lln}, uniformly for $y' \in [Y_0, Y_0 + \theta]$
\begin{equation*}
q^{H(L(X_0 + \theta), L y')} - q^{H(LX_0, Ly')} =  \mathfrak{q}^{\mathbf{h}(X_0 + \theta, y')} - \mathfrak{q}^{\mathbf{h}(X_0, y')} + o(1).
\end{equation*}
Consequently, it follows from \eqref{eq:temp19}
\begin{align*}
&L^{-1} \sum_{\substack{x \in [LX_0, L(X_0 + \theta)]\\ y \in [LY_0, L(Y_0 + \theta)]}} \mathcal{R}^d (LX, LY, x+1, y+1)^2 q^{H(x, y)} \Delta_x\\ 
&= L^{-1} \mathcal{R}_{IJ}(X, Y, X_0, Y_0)^2 \mathfrak{q}^{\mathbf{h}(X_0, Y_0)} \int_{Y_0}^{Y_0 + \theta} \Big(\mathfrak{q}^{\mathbf{h}(X_0 + \theta, y)} - \mathfrak{q}^{\mathbf{h}(X_0, y)}\Big) dy + \theta o(1) + \mathcal{O}(\theta^3) + \theta^2 o(1),
\\
\numberthis \label{eq:temp30}
&= L^{-1} \mathcal{R}_{IJ}(X, Y, X_0, Y_0)^2 \mathfrak{q}^{\mathbf{h}(X_0, Y_0)} \mathfrak{q}_x^{\mathbf{h}(X_0, Y_0)} \theta^2 + \theta o(1) + \mathcal{O}(\theta^3) + \theta^2 o(1).
\end{align*}
Note that in the last line, we used the property that the solution $\mathfrak{q}^{\mathbf{h}}$ to the \eqref{eq:llntelegraph} is piecewise $C^1$ (since we assume additionally the boundary $\chi$ and $\psi$ are smooth). By \eqref{eq:temp30}, 
\begin{align*}
&L^{-1} \sum_{x = 0}^{LX - 1} \sum_{y = 0}^{LY -1} \mathcal{R}^d (LX, LY, x+1, y+1)^2 q^{H(x, y)} \Delta_x\\
\numberthis \label{eq:temp23}
&= \sum_{0 \leq i \ \leq X/\theta} \sum_{0 \leq j \leq Y/\theta}  \mathcal{R}_{IJ} (X, Y, \theta i, \theta j)^2 \mathfrak{q}^{\mathbf{h}(\theta i, \theta j)} \mathfrak{q}_x^{\mathbf{h}(\theta i, \theta j)} \theta^2 + (1 + \theta^{-1}) o(1) + \mathcal{O}(\theta).
\end{align*}
By first letting $L \to \infty$ then $\theta \to 0$, we conclude the desired \eqref{eq:gradientx}.
The approximation for \eqref{eq:gradienty} is similar, we omit the detail. 
\bigskip
\\
Things become more involved for \eqref{eq:gradientxy}, note that
\begin{align*}
&\sum_{x = LX_0}^{L(X_0 + \theta)} \sum_{y = LY_0 }^{L(Y_0 + \theta)} \mathcal{R}^d(LX, LY, x+1, y+1)^2 \Delta_x \Delta_y\\ 
&= \mathcal{R}^d (LX, LY, LX_0, Y_0) \bigg(\sum_{x = LX_0}^{L(X_0 + \theta)} \sum_{y = LY_0 }^{L(Y_0 + \theta)}  \Delta_x \Delta_y\bigg) + \mathcal{O}(\theta^3),\\
\numberthis \label{eq:temp21}
&= L^{-2} (\log\mathfrak{q})^2 \mathcal{R}_{IJ} (X, Y, X_0, Y_0) \mathfrak{q}^{2\mathbf{h}(X_0, Y_0)} 
\bigg(\sum_{x = LX_0}^{L(X_0 + \theta)}  \sum_{y = LY_0}^{L(Y_0 + \theta)}\nabla_x H(x, y) \nabla_y H(x, y)\bigg) + o(1) \mathcal{O}(\theta^2) + \mathcal{O}(\theta^3), 
\end{align*}
where we denote by $\nabla_x H(x, y) := H(x+1, y) - H(x, y), \nabla_y H(x, y) := H(x, y+1) - H(x, y).$
In the last equality, we used the approximation in \eqref{eq:temp20} and 
\begin{align*}
\Delta_x &= q^{H(x+1, y)} - q^{H(x, y)} = L^{-1} \nabla_x H(x, y) \mathfrak{q}^{\mathbf{h}(X_0, Y_0)} \log \mathfrak{q} + L^{-1} o(1),\\ 
\Delta_y &= q^{H(x, y+1)} - q^{H(x, y)} = L^{-1} \nabla_y H(x, y) \mathfrak{q}^{\mathbf{h}(X_0, Y_0)} \log \mathfrak{q} + L^{-1} o(1) .
\end{align*}
Note that $-\nabla_x H(x, y), \nabla_y H(x, y)$ indicate the number of lines entering into the vertex $(x, y)$ from bottom and left.
\bigskip
\\
For a vertex associated with four tuple $(i_1, j_1; i_2, j_2)$, we say this vertex is \emph{unusual} if $i_1 \neq i_2$ or $j_1 \neq i_2$. Let $\square$ denote the square $[L X_0, L X_0 + L\theta] \times [LY_0, LY_0 + L \theta]$ and suppose that there are respectively $n$ and $m$ lines entering inside $\square$ from bottom and left. Let $\mathcal{C}$ be the number of unusual vertices in the square. If $\mathcal{C} = 0$, it is clear that  
\begin{equation*}
\sum_{x = LX_0}^{L(X_0 + \theta)} \sum_{y = LY_0}^{L(Y_0 + \theta)} \nabla_x H(x, y) \nabla_y H(x, y) = -n m. 
\end{equation*}
Each unusual vertex might change the LHS summation at most by $2IJ \theta L$. As an analogue of \cite[Eq. 93]{BG19},
\begin{equation*}
\Big|\sum_{x = LX_0}^{L(X_0 + \theta)} \sum_{y = LY_0}^{L(Y_0 + \theta)} \nabla_x H(x, y) \nabla_y H(x, y) + n m \Big| \leq IJ \theta L \cdot \mathcal{C}.
\end{equation*}
It follows from Lemma \ref{lem:scalingmatrix} that the probability that a vertex is unusual is upper bounded by $CL^{-1}$, where $C$ is a constant. Thus, 
\begin{equation}\label{eq:temp22}
\Big|\sum_{x = LX_0}^{L(X_0 + \theta)} \sum_{y = LY_0}^{L(Y_0 + \theta)} \nabla_x H(x, y) \nabla_y H(x, y) + n m \Big| \leq \text{const} \cdot \theta^3 L^2,
\end{equation}
with high probability as $L \to \infty$.
Noting that 
\begin{equation*}
H\big(L(X_0 +\theta), L Y_0\big) - H(LX_0, LY_0) = -n, \qquad  H\big(L X_0, L (Y_0 + Y)\big) - H(LX_0, LY_0) = m.
\end{equation*}
Combining \eqref{eq:temp21} and \eqref{eq:temp22} (together with Theorem \ref{thm:fourpoint}) yields
\begin{align*}
&\sum_{x = LX_0}^{L(X_0 + \theta)} \sum_{y = LY_0}^{L(Y_0 + \theta)} \mathcal{R}^d(LX, LY, x+1, y+1)^2 \Delta_x \Delta_y\\ &= L^{-2} (\log \mathfrak{q})^2 \mathcal{R}_{IJ} (X, Y, X_0, Y_0) \mathfrak{q}^{2 \mathbf{h}(X_0, Y_0)} \big(H(L(X_0 + \theta), L Y_0) - H(LX_0, LY_0)\big) \big(H(L X_0, L (Y_0 + \theta)\big) - H(LX_0, LY_0)\big)\\
&\quad + o(1) \mathcal{O}(\theta^2) + \mathcal{O}(\theta^3)\\
&= \mathcal{R}_{IJ} (X, Y, X_0, Y_0) (\log \mathfrak{q})^2 \mathfrak{q}^{2 \mathbf{h}(X_0, Y_0)} \big(\mathbf{h}(X_0 + \theta, Y_0) - \mathbf{h}(X_0, Y_0)\big) \big(\mathbf{h}(X_0,  Y_0 + \theta) - \mathbf{h}(X_0, Y_0)\big) + o(1) \mathcal{O}(\theta^2) + \mathcal{O}(\theta^3)
\end{align*} 
Using similar approximation as in \eqref{eq:temp23}, by first letting $L \to \infty$ then $\theta \to 0$, we demonstrate \eqref{eq:gradientxy}. Having proved \eqref{eq:gradientx}-\eqref{eq:gradientxy}, we simply obtain the desired \eqref{eq:temp12}.
\bigskip
\\
We conclude the theorem using martingale CLT \cite[Section 3]{HH14}. Recall that $$M_L (t) = \sqrt{L} \sum_{i = 1}^t \mathcal{R}^d\big(LX, LY, x(i), y(i)\big) \xi\big(x(i), y(i)\big),\qquad  t \in [1, L^2 X Y],$$ 
We want to show $U_L (X, Y) =  M_L(L^2 XY) \to \varphi(X, Y)$ in law as $L \to \infty$. By Theorem \ref{thm:fourpoint}, $M_L (t)$ is a martingale  with respect to the its own filtration. The proof of Theorem \ref{thm:clt} reduces to verify the following conditions for martingale CLT:
\begin{enumerate}[(i).]
\item The conditional covariance of $M_L (t)$ at $t = L^2 XY$, which equals 
\begin{equation*}
L \sum_{x= 0}^{LX-1} \sum_{y = 0}^{LY-1} \mathcal{R}^d (LX, LY; x+1, y+1)^2 \mathbb{E}\Big[\xi(x+1, y+1)^2 \big| \mathcal{F}(x,y)\Big],
\end{equation*}
has the same $L \to \infty$ behavior as its unconditional variance, in the sense that their ratio tends to $1$ in probability. 
\item The Lindeberg's condition, i.e. $\lim_{L \to \infty} \sum_{i = 1}^{L^2 XY} \mathbb{E}\Big[(M_L (i) - M_L (i-1))^2 \mathbf{1}_{\{(M_L (i) - M_{L} (i-1))^2 > \epsilon\}}\Big] = 0$.	
\end{enumerate}
Using Corollary \ref{cor:quadbound}, it is clear that the conditional variance on the LHS of \eqref{eq:temp12} is uniformly bounded. By the convergence in \eqref{eq:temp12} together with dominated convergence theorem, we know that both the conditional and unconditional variance of $M_L (t)$ at $t = L^2 XY$ converge to the RHS of \eqref{eq:temp12} (which equals to variance of $\varphi(X, Y)$ given in Remark \ref{rmk:covariance}), this concludes (i).
\bigskip
\\
The Lindeberg's condition (ii) follows directly from how $\xi$ is defined:
By straightforward computation, there exists constant $C$ such that $|\xi(x+1, y+1)| \leq C L^{-1}$  for all $x \in [0, LA]$ and $y \in [0, LB]$. In addition,  $\mathcal{R}^d (LX, LY, x, y)$ is uniformly bounded. So when $L$ is large enough,
$$\Big\{\big(M_L (i) - M_L (i-1)\big)^2 > \epsilon\Big\} = \Big\{L \mathcal{R}^d (LX, LY, x(i), y(i))^2  \xi(x(i), y(i))^2 > \epsilon\Big\} = \emptyset,$$ 
which implies that for every $i \in [1, L^2 XY]$,
\begin{equation*}
\mathbb{E}\Big[(M_L (i) - M_L (i-1))^2 \mathbf{1}_{\{(M_L(i) - M_L (i-1))^2 > \epsilon\}}\Big] = 0.
\end{equation*}
Having verified (i) and (ii), we conclude our proof using the martingale central limit theorem.
\end{proof}
We move on proving Proposition \ref{prop:tightness}. Before presenting our proof,  we require the following result. 
\begin{lemma}\label{lem:technical bound}
Fixed $A, B \geq 0$ and $n, \ell_1, \dots \ell_n \in \mathbb{N}$, there exists constant $C$ (only depends on $A, B, n$) such that for all $L > 1$ and arbitrary distinct points $(x_i, y_i) \in [1, LA] \times [1, LB]$, $i = 1, \dots n$,
\begin{equation*}
\mathbb{E}\Big[\prod_{i=1}^n |\xi(x_i, y_i)|^{\ell_i}\Big] \leq C L^{-\sum_{i=1}^n (\ell_i + 1)}.
\end{equation*}
\end{lemma}
\begin{proof}
	
It suffices to prove that for $(x, y) \in [0, LA-1] \times [0, LB-1]$,
\begin{equation}\label{eq:temp35}
\mathbb{E}\Big[|\xi(x+1, y+1)|^{\ell} \big|  \mathcal{F}(x, y)\Big] \leq C L^{-\ell - 1}.
\end{equation}
We first finish the proof of the lemma by assuming \eqref{eq:temp35}. Consider the ordering \eqref{eq:ordering} of integer points in $[1, LA] \times [1, LB]$, without loss of generality, we  assume $(x_i, y_i) = (x(s_i), y(s_i))$ so that $s_1 <  \dots  < s_n$. Recall that $\mathcal{F}(x, y) =  \sigma\big(H(i, j): i \leq x \text{ or } j \leq y\big)$, so $\xi(x_i, y_i) \in \mathcal{F}(x_n-1, y_n-1)$ for $i = 1, \dots , n-1$. By \eqref{eq:temp35} and conditioning,
\begin{align*}
\mathbb{E}\Big[\prod_{i = 1}^n|\xi(x_i, y_i)|^{\ell_i}\Big] &= \mathbb{E}\Big[\prod_{i=1}^{n-1} |\xi(x_i, y_i)|^{\ell_i}\Big] \mathbb{E}\Big[|\xi(x_n, y_n)|^{\ell_n} \big| \mathcal{F}(x_n - 1, y_n - 1)\Big]\\
&\leq C L^{-\ell_n -1} \mathbb{E}\Big[\prod_{i=1}^{n-1} |\xi(x_i, y_i)|^{\ell_i}\Big].
\end{align*}
Iterating the above inequality, we conclude the lemma.
\bigskip
\\
We move on showing \eqref{eq:temp35}. Denote $v, v'$ to be the vertical input and output for the vertex $(x, y)$ and $h$ to be the horizontal input, i.e. 
$$v:= H(x, y) - H(x+1, y),\quad v':= H(x, y+1) - H(x+1, y+1),\quad  h:= H(x, y+1) - H(x, y)  .$$ It is evident that we can rewrite $\xi(x+1, y+1)$ in \eqref{eq:xi} as
\begin{align}\label{eq:temp40}
\xi(x+1, y+1) = q^{H(x, y)} \big(q^{h-v'} - b_1 q^h - b_2 q^{-v} + b_1 + b_2-1\big),
\end{align}
recall $b_1 = \frac{\alpha + \nu}{1 + \alpha}$ and $b_2 = \frac{1 + \alpha q^J}{1 + \alpha}$. Since $q = \mathfrak{q}^{\frac{1}{L}}$ where $\mathfrak{q}$ is fixed, 
so for $(x, y) \in [0, LA] \times [0, LB]$, there exists $C$ such that $\frac{1}{C} \leq q^{H(x, y)} \leq C$. In addition, by  \eqref{eq:scalingb1b2},
\begin{align*}
q^{h-v'} - b_1 q^{h} - b_2 q^{-v} + b_1 + b_2 - 1 &= q^{h-v'} - q^{h} - q^{-v} + 1 + (1-b_1) (q^h - 1) + (1-b_2) (q^{-v} - 1)\\  
&= \ln \mathfrak{q} (v - v') L^{-1} + \mathcal{O}(L^{-2})
\end{align*}
Referring to \eqref{eq:temp40}, we conclude that for fixed $A$ and $B$ there exists a constant $C$ such that for arbitrary $L > 1$, $(x, y) \in [0, LA] \times [0, LB]$, 
\begin{equation}\label{eq:boundforxi}
\begin{aligned}
&|\xi(x+1, y+1)| \leq C L^{-2} \qquad \text{if } (h, v) = (h', v')\\
&|\xi(x+1, y+1)| \leq C L^{-1} \qquad \text{if } (h, v) \neq  (h', v')
\end{aligned}
\end{equation}
Note that 
\begin{align*}
\mathbb{E}\Big[|\xi(x+1, y+1)|^\ell \big| \mathcal{F}(x, y)\Big]
&= \mathbb{E}\Big[|\xi(x+1, y+1)|^{\ell} \big| \sigma(H(x, y), h, v)\Big]\\
\numberthis \label{eq:temp41}
&= \sum_{(h', v'): h'+v' = h + v} L_{\alpha}^{(J)}(h, v; h', v') |\xi(x+1, y+1)|^\ell
\end{align*}
Using Lemma \ref{lem:scalingmatrix} and \eqref{eq:boundforxi}, we know that for each term in the summation: Either $(h', v') \neq (h, v)$, which implies  
$L_{\alpha}^{(J)} (h, v; h', v') \leq C L^{-1}$ and $|\xi(x+1, y+1)| \leq C L^{-1}$; Either $(h, v) = (h', v')$, which yields $|\xi(x+1, y+1)| \leq C L^{-2}$. Hence, the absolute value for each term in the summation \eqref{eq:temp41} is upper bounded by $C L^{-\ell -1}$. As the summation is finite,  we conclude \eqref{eq:temp35}.
\end{proof}
\begin{proof}[Proof of Proposition \ref{prop:tightness}]
Using the Kolmogorov-Chentsov criterion, the tightness of $U_L (\cdot, \cdot)$ follows directly from \eqref{eq:incrementbound}. To prove \eqref{eq:incrementbound}, it suffices to show that there exists constant $C$ such that  for $X \in [0, LA]$ and $0\leq  Y_1 \leq Y_2 \leq LB$,
\begin{align}\label{eq:temp42}
\mathbb{E}\Big[\Big(U_L(X, Y_1) - U_L(X, Y_2)\Big)^{2n}\Big] \leq C |Y_1 - Y_2|^n,
\end{align}
Since we linearly interpolate $H(X, Y)$ for non-integer $X, Y$ and $U_L(X, Y)$ is expressed in terms of $H(LX, LY)$, we can assume $Y_2 - Y_1 \geq L^{-1}.$
Referring to \eqref{eq:temp10}, we know that 
\begin{equation}\label{eq:temp39}
U_L(X, Y) = \sqrt{L} \sum_{x=1}^{LX} \sum_{y=1}^{LY} \mathcal{R}^d(LX, LY, x, y) \xi(x, y),
\end{equation}
which implies 
\begin{align*}
U_L (X, Y_2) - U_L(X, Y_1) &=  \sum_{x=1}^{LX} \sum_{y=1}^{LY_1} \sqrt{L} \big(\mathcal{R}^d(LX, LY_1, x, y) - \mathcal{R}^d(LX, LY_2, x, y)\big) \xi(x, y)\\
&\quad + \sum_{x=1}^{LX}\sum_{y = LY_1 + 1}^{LY_2} \sqrt{L}\mathcal{R}^d(LX, LY_2, x, y) \xi(x, y)
\end{align*} 
Taking the $n$-th power of both sides in the above display and apply the inequality $(a+b)^{2n} \leq 2^{2n-1} (a^{2n} + b^{2n})$ to the RHS, we have 
\begin{align*}
\mathbb{E}\Big[\big(U_L(X_, Y_2) - U_L (X, Y_1)\big)^{2n}\Big] &\leq 2^{2n-1}
\mathbb{E}\Big[\Big(\sum_{x=1}^{LX} \sum_{y=1}^{LY_1} \sqrt{L} \big(\mathcal{R}^d(LX, LY_1, x, y) - \mathcal{R}^d(LX, LY_2, x, y)\big) \xi(x, y)\Big)^{2n}\Big]\\ 
\numberthis \label{eq:temp33}
&\quad + 2^{2n-1} \mathbb{E}\Big[\Big( \sum_{x=1}^{LX} \sum_{y = LY_1 + 1}^{LY_2} \sqrt{L}\mathcal{R}^d(LX, LY_2, x, y) \xi(x, y)\Big)^{2n}\Big]
\end{align*}
Denote the first and second term above (without the constant multiplier) by $\mathbf{M}_1$ and $\mathbf{M}_2$ respectively. We proceed to upper bound $\mathbf{M}_1$ and $\mathbf{M}_2$ respectively.
\bigskip
\\
For $\mathbf{M}_1$, since $\xi(x, y)$ is a martingale increment, by Burkholder–Davis–Gundy inequality, we have  
\begin{align*}
\mathbf{M}_1 \leq
C(n) L^n \mathbb{E}\bigg[\Big(\sum_{x=1}^{LX} \sum_{y=1}^{LY_1} \big(\mathcal{R}^d (LX, LY_1, x, y) - \mathcal{R}^d(LX, LY_2, x, y)\big)^2 \xi(x, y)^2\Big)^{n}\bigg],
\end{align*}
where the constant $C(n)$ only depends on $n$. Under scaling \eqref{eq:scalingb1b2}, there exists a constant $C$ such that for $L> 1$, $X \in [0, LA]$ and $Y_1, Y_2 \in [0, LB]$ 
(one can see this from the expression of $\mathcal{R}^d$ in \eqref{eq:discreteRiemann}),
\begin{equation*}
\big|\mathcal{R}^d(LX, LY_1, x, y) - \mathcal{R}^d(LX, LY_2, x, y)\big| \leq C |Y_1 - Y_2|,
\end{equation*} 
this implies 
\begin{equation}\label{eq:M1bound}
\mathbf{M}_1 \leq C(n) |Y_1  -Y_2|^{2n}\cdot L^{n}  \mathbb{E}\bigg[\Big(\sum_{x=1}^{LX} \sum_{y=1}^{LY_1} \xi(x, y)^2\Big)^n\bigg].
\end{equation}
We claim that for all $L > 1$, the term 
$L^n \mathbb{E}\Big[\big(\sum_{x=1}^{LX} \sum_{y=1}^{LY_1} \xi(x, y)^2\big)^n\Big]$ is uniformly upper bounded for $(x, y) \in [0, LA] \times [0, LB]$. To see this, we expand the $n$-th power of the double summation in the expectation above. It is not hard to see that there exists a constant $C$ such that 
\begin{align*}
L^n \mathbb{E}\bigg[\Big(\sum_{x=1}^{LX} \sum_{y=1}^{LY_1} \xi(x, y)^2\Big)^n\bigg]
&\leq C L^n \sum_{\lambda\vdash n}\sum_{\substack{(x_i, y_i) \in [1, LX] \times [1, LY_1]\\ i = 1, \dots , \ell(\lambda), (x_i, y_i)  \text{ are distinct}}} \mathbb{E}\Big[\prod_{i=1}^{\ell(\lambda)} \xi(x_i, y_i)^{2\lambda_i}\Big]
\end{align*}
Here, the summation above is taken over the partition $\lambda$ of $n$, that is to say, $\lambda = (\lambda_1 \geq \dots \geq \lambda_{s}) \in \mathbb{Z}_{\geq 1}^{s}$ with $\sum_{i=1}^{s} \lambda_i = n$, $\ell(\lambda) = s$ is the length of the partition $\lambda$. We want to upper bound the right hand side in the above display. By Lemma \ref{lem:technical bound}, we know that the $\mathbb{E}\Big[\prod_{i=1}^{\ell(\lambda)} \xi(x_i, y_i)^{2\lambda_i}\Big]$ can be upper bounded by a constant times $L^{-2n - \ell(\lambda)}$. In addition, it is clear that $\#\big\{(x_i, y_i) \in [1, LX] \times [1, LY_1], i = 1, \dots , \ell(\lambda), (x_i, y_i)  \text{ are distinct}\big\} \leq (L^2 X Y_1)^{\ell(\lambda)}$ ($\#A$ denotes the number of elements in the set $A$). Consequently 
$$L^n \mathbb{E}\bigg[\Big(\sum_{x=1}^{LX} \sum_{y=1}^{LY_1} \xi(x, y)^2\Big)^n\bigg] \leq C L^n \sum_{\lambda \vdash n} (L^2 X Y_1)^{\ell(\lambda)} L^{-2n-\ell(\lambda)} \leq C.$$ 
Inserting the above upper bound into \eqref{eq:M1bound} implies 
\begin{equation}\label{eq:temp37}
\mathbf{M}_1 \leq C(n) |Y_2 - Y_1|^{2n}.
\end{equation} 
We proceed to upper bound $\mathbf{M}_2$. Again, using Burkholder-Davis-Gundy inequality, one obtains 
\begin{equation*}
\mathbf{M}_2 \leq C(n) L^n \mathbb{E}\bigg[\Big(\sum_{x = 1}^{LX} \sum_{y = LY_1 + 1}^{LY_2} \mathcal{R}^d(LX, LY_2, x, y)^2 \xi(x, y)^2\Big)^n\bigg].
\end{equation*}
Expanding the $n$-th power for the double summation and upper bounding the square of $\mathcal{R}^d$ by a constant,  
\begin{align*}
\mathbf{M}_2 &\leq C(n) L^n \sum_{\lambda\vdash n}\sum_{\substack{(x_i, y_i) \in [1, LX] \times (LY_1, LY_2]\\ i = 1, \dots , \ell(\lambda), (x_i, y_i)  \text{ are distinct}}} \mathbb{E}\Big[\prod_{i=1}^{\ell(\lambda)} \xi(x_i, y_i)^{2\lambda_i}\Big].
\end{align*}
Using Lemma \ref{lem:technical bound} and by similar argument for upper bounding $\mathbf{M}_1$, we have 
\begin{align}\label{eq:temp36}
\mathbf{M}_2 \leq C(n) L^n \sum_{\lambda \vdash n} (L^2 X (Y_2 - Y_1))^{\ell(\lambda)} L^{-2n-\ell(\lambda)} \leq  C(n) L^{\ell(\lambda) - n} (Y_2 -Y_1)^{\ell(\lambda)} \leq C(n) |Y_2 - Y_1|^{n}  
\end{align}
The last inequality in the above display is due to our assumption $Y_2 - Y_1 \geq L^{-1}$. 
\bigskip
\\
Referring to \eqref{eq:temp33}, we have 
$$\mathbb{E}\Big[\big(U_L(X_, Y_2) - U_L (X, Y_1)\big)^{2n}\Big] \leq 2^{2n-1} (\mathbf{M}_1 + \mathbf{M}_2).$$
Combining \eqref{eq:temp37} with \eqref{eq:temp36}, we conclude \eqref{eq:temp42}.
\end{proof}
\begin{remark}\label{rmk:tightness}
It is worth remarking that the classical theory for martingale functional CLT, e.g. \cite[Section 6]{Bro71}, might not be helpful for proving our tightness. In order to get the tightness, the classical theory requires $U_L (X, Y)$ to be a martingale in $(X, Y)$ in order to control (using martingale inequalities) the modulus $$\sup_{|X_1 - X_2| + |Y_1 - Y_2| \leq \delta} |U_L(X_1, Y_1) - U_L(X_2, Y_2)|,$$
	for small $\delta > 0$, and then apply the Arzela-Ascoli. See \cite[Theorem 7.3]{Bil13}. In our case, though $\xi(x, y)$ is a martingale increment,  $U_L(X, Y)$ fails to be a martingale due to dependence of $\mathcal{R}^d$ on $X, Y$ in \eqref{eq:temp39}. 
\end{remark}
\begin{proof}[Proof of Corollary \ref{cor:clt}]
It suffices to prove the weak convergence for arbitrary interval $[0, A ]\times [0, B]$. Note that $U(x, y) = q^{H(x, y)} - \mathbb{E}\big[q^{H(x, y)}\big]$, then
\begin{align}\label{eq:temp14}
H(Lx, Ly) = L \log_\mathfrak{q} \Big(q^{H(Lx, Ly)}\Big)  = L \log_\mathfrak{q} \mathbb{E}\Big[q^{H(Lx, Ly)}\Big] + L \log_\mathfrak{q} \bigg(1 + \frac{U(Lx, Ly)}{\mathbb{E}\big[q^{H(Lx, Ly)}\big]}\bigg).
\end{align}
Since $H(x, y)$ is Lipschitz and $q = \mathfrak{q}^{\frac{1}{L}}$ (where $\mathfrak{q}$ is fixed), there exists constant $C$ such that for $(x, y) \in [0, LA] \times [0, LB]$, 
\begin{equation*}
\frac{1}{C} \leq q^{H(Lx, Ly)} \leq C,  \qquad
\frac{1}{C} \leq \mathbb{E}\Big[q^{H(Lx, Ly)}\Big] \leq C.
\end{equation*} 
For the second term on the right hand side of \eqref{eq:temp14},  we taylor expand the function $\log_{\mathfrak{q}} (1+x)$ around $x = 0$,  
\begin{align*}
H(Lx, Ly) = L \log_{\mathfrak{q}} \mathbb{E}\Big[q^{H(Lx, Ly)}\Big] +  \frac{LU(Lx, Ly)}{\log \mathfrak{q} \cdot \mathbb{E}\big[q^{H(Lx, Ly)}\big]} + L r_L(x, y),
\end{align*} 
where $|r_L(x, y)| \leq C U(Lx, Ly)^2 / \big(\mathbb{E}[q^{H(Lx, Ly)}]\big)^2 \leq C U(Lx, Ly)^2$. Consequently, since $\mathbb{E}\big[U(Lx, Ly)\big] = 0$,
\begin{equation*}
\frac{H(Lx, Ly) - \mathbb{E}\Big[H(Lx, Ly)\Big]}{\sqrt{L}} = \frac{\sqrt{L} U(Lx, Ly)}{\mathbb{E}\Big[q^{H(Lx, Ly)}\Big] \log \mathfrak{q}} + \sqrt{L} \Big(r_L (x, y) - \mathbb{E}\big[r_L (x, y)\big]\Big).
\end{equation*}
By Proposition \ref{prop:tightness}, we know that  $U_L (\cdot, \cdot) = \sqrt{L} U(L\cdot, L \cdot)$ is tight. Thus, for any fixed $A, B > 0$, as $L \to \infty$, $$\sup_{x \in [0, A] \times [0,B]} L^{\frac{1}{2}} U (Lx, Ly)^2 \to 0\quad \text{ in probability. }$$
Since $|r_L (x, y)| \leq C U(Lx, Ly)^2$, 
$$\sup_{(x, y) \in [0, A] \times [0, B]}\sqrt{L}\Big|r_L (x, y) - \mathbb{E}\big[r_L (x, y)\big]\Big| \to 0 \quad \text{ in probability.}$$
Therefore, we have the weak convergence in $C([0, A] \times [0, B])$,
\begin{align*}
\lim_{L \to \infty} \frac{H(Lx, Ly) - \mathbb{E}\Big[H(Lx, Ly)\Big]}{\sqrt{L}} = \lim_{L \to \infty} \frac{\sqrt{L} U(Lx, Ly)}{\mathbb{E}\big[q^{H(Lx, Ly)}\big]\log \mathfrak{q}} = \frac{\varphi(x, y)}{\mathfrak{q}^{\mathbf{h}(x, y)} \log \mathfrak{q}}.
\end{align*}
To get the second equality above, we apply Theorem \ref{thm:lln} and Theorem \ref{thm:clt} to the denominator and numerator respectively. By straightforward computation,  $\phi(x, y) := \frac{\varphi(x, y)}{\mathfrak{q}^{\mathbf{h}(x, y)} \log \mathfrak{q}}$ solves \eqref{eq:temp11}, which concludes the corollary.
\end{proof}

\bibliographystyle{alpha}
\bibliography{higherspintotelegraph.bib}
\end{document}